\RequirePackage{fix-cm}
\documentclass[smallextended,numbook,runningheads]{svjour3}
\smartqed  
\usepackage{graphicx,graphics}

\usepackage{amsmath,amssymb}
\usepackage[utf8]{inputenc}

\usepackage{epstopdf}

\usepackage{enumerate}
\usepackage{mathtools}
\usepackage{color,xcolor} 
\usepackage{mathrsfs,bbm} 

\journalname{}
\date{ \phantom{b} \vspace{45mm}\phantom{e}}
\def\makeheadbox{\hspace{-4mm}Version of 17 January 2022}

\newcommand{\Z}{\mathbb{Z}}

\newcommand{\R}{\mathbb{R}}
\newcommand{\C}{\mathbb{C}}

\newcommand{\mXG}{{ \mathrm{\boldsymbol{X}}_\Gamma}}
\newcommand{\mVG}{{\mathrm{\boldsymbol{V}}_\Gamma}}

\renewcommand{\Re}{\mathrm{Re}\,}

\newcommand{\pt}{\partial_t}
\newcommand{\pttau}{\partial_t^\tau}

\newcommand{\sg}{\nabla_\Gamma}

\newcommand{\ginc}{\boldsymbol{g}^{\textnormal{inc}}}
\newcommand{\Einc}{\boldsymbol{E}^\textnormal{inc}}

\newcommand{\Bimp}{\Cald_\textnormal{imp}}

\newcommand{\boldsymb}{\boldsymbol }

\newcommand{\bpsi}{\boldsymbol{\psi}}
\newcommand{\bphi}{\boldsymbol{\phi}}
\newcommand{\bvar}{\boldsymbol{\varphi}}

\newcommand{\Om}{\Omega}

\newcommand{\wginc}{\widehat{\g}^{\textnormal{inc}}}

\newcommand{\wvarphi}{\widehat{\bvar}}
\newcommand{\wpsi}{\widehat{\bpsi}}

\newcommand{\divG}{\operatorname{div}_\Gamma}
\DeclareMathOperator{\curl}{\mathbf{curl}}

\newcommand{\norm}[1]{\left\lVert#1\right\rVert}

\newcommand{\abs}[1]{\left|#1\right|}
\newcommand{\jmp}[1]{[#1]}
\newcommand{\avg}[1]{\{#1\}}

\renewcommand{\Re}{\operatorname{Re}}
\newcommand{\dist}{\operatorname{dist}}
\newcommand{\inc}{^\textnormal{inc}}
\newcommand{\scat}{^\textnormal{scat}}
\newcommand{\tot}{^\textnormal{tot}}
\newcommand{\normal}{\boldsymbol{\nu}}

\newcommand{\gaT}{\boldsymb{\gamma}_{T}}

\newcommand{\A}{\boldsymb A}

\newcommand{\B}{\boldsymb H}
\newcommand{\Bd}{\boldsymb B}
\newcommand{\bC}{\boldsymb C}
\newcommand{\E}{\boldsymb E}
\newcommand{\Id}{\boldsymb{I\!\hspace{0.7pt}d}}
\renewcommand{\H}{\boldsymb H}
\newcommand{\K}{\boldsymb K}
\newcommand{\bL}{\boldsymb L}
\newcommand{\bR}{\boldsymb R}
\newcommand{\U}{\mathcal U}
\newcommand{\V}{\boldsymb V}
\newcommand{\W}{\mathcal W}
\newcommand{\X}{\boldsymb X}
\renewcommand{\Z}{\boldsymb Z}

\newcommand{\Cald}{\boldsymb C}

\renewcommand{\a}{\boldsymb a}
\renewcommand{\b}{\boldsymb b}
\newcommand{\ej}{\boldsymbol{e_{j}}}
\renewcommand{\b}{\boldsymb b}
\newcommand{\g}{\boldsymb g}
\newcommand{\p}{\boldsymb p}
\renewcommand{\u}{\boldsymb u}
\renewcommand{\v}{\boldsymb v}

\newcommand{\x}{\boldsymb x}
\newcommand{\y}{\boldsymb y}

\newcommand{\cqb}{\boldsymb b}
\newcommand{\cqc}{\boldsymb c}
\newcommand{\cqg}{\boldsymb g}
\newcommand{\cqC}{\boldsymb C}
\newcommand{\cqK}{\boldsymb K}
\newcommand{\cqH}{\boldsymb H}
\newcommand{\cqI}{\boldsymb I}
\newcommand{\cqL}{\boldsymb L}
\newcommand{\cqT}{\boldsymb T}
\newcommand{\cqW}{\boldsymb W}
\newcommand{\cqX}{\boldsymb X}
\newcommand{\cqY}{\boldsymb Y}

\newcommand{\bupsilon}{\boldsymb \upsilon}
\newcommand{\bxi}{\boldsymb \xi}

\newcommand{\Ga}{\Gamma}
\newcommand{\eps}{\varepsilon}
\newcommand{\bone}{\mathbbm{1}}
\newcommand{\Ctrace}{C_\Gamma}
\newcommand{\ctrace}{c_\Gamma}

\begin{document}

\title{
Time-dependent electromagnetic scattering from thin layers
}

\titlerunning{Time-dependent electromagnetic scattering from thin layers}        

\author{J\"org Nick \and Bal\'azs Kov\'acs \and Christian~Lubich}

\authorrunning{J.~Nick, B.~Kov\'acs, and Ch.~Lubich} 

\institute{J\"org Nick and Christian Lubich \at
		Mathematisches Institut, Universit\"at T\"{u}bingen,\\
		Auf der Morgenstelle 10, 72076 T\"{u}bingen, Germany \\
		\email{\{nick,lubich\}@na.uni-tuebingen.de} 
		\and
		Bal\'azs Kov\'acs \at
		Faculty of Mathematics, University of Regensburg, \\
		Universit\"atsstra\ss{}e 31, 93049 Regensburg, Germany \\
		\email{balazs.kovacs@mathematik.uni-regensburg.de}
}


\date{ }
\maketitle

\begin{abstract}
	The scattering of electromagnetic waves from obstacles with wave-material interaction in thin layers on the surface is described by generalized impedance boundary conditions, which provide effective approximate models.
	In particular, this includes a thin coating around a perfect conductor and the skin effect of a highly conducting material. The approach taken in this work is to derive, analyse and discretize a system of time-dependent boundary integral equations that determines the tangential traces of the scattered electric and magnetic fields. In a familiar second step, the fields are evaluated in the exterior domain by a representation formula, which uses the time-dependent potential operators of Maxwell's equations.
	The  time-dependent boundary integral equation 
	is discretized with Runge--Kutta based convolution quadrature in time and Raviart--Thomas boundary elements in space. 
	 
	Using the frequency-explicit bounds from the well-posedness analysis given here together with known approximation properties of the numerical methods, the full discretization is proved to be stable and convergent, with explicitly given rates in the case of sufficient regularity. Taking the same
	Runge--Kutta based convolution quadrature for discretizing the time-dependent representation formulas, the optimal order of convergence is obtained away from the scattering boundary, whereas an order reduction occurs close to the boundary.
	
	The theoretical results are illustrated by numerical experiments.

\end{abstract}

\section{Introduction}
This work studies a numerical approach to computing time-domain electromagnetic scattering from obstacles that, due to their material properties, involve multiple scales and yield effective boundary conditions known as generalized impedance boundary conditions.

\subsection{ Time-dependent Maxwell's equations in an exterior domain}
On an exterior Lipschitz domain $\Omega$, which is 
the complement of one or multiple bounded domains, we consider the \emph{time-dependent Maxwell's equations} for the total electric field $\E\tot(\x,t)$ and the total  magnetic  field $\H\tot(\x,t)$,
\begin{align}\label{maxwell-tot}
\begin{split}
\varepsilon\,\pt \E\tot-\curl \H\tot &= 0 \quad  \\
\mu \,\pt \H\tot + \curl \E\tot &=0
\end{split} \quad \text{ in the exterior domain }\Omega.
\end{align}
The permittivity $\varepsilon$ and the permeability $\mu$ are taken as positive constants in~$\Omega$.

 Our interest here is when Maxwell's equations are coupled with nontrivial boundary conditions that describe scattering.  
We assume to be given incident electric and magnetic fields $(\E\inc,\B\inc)$, which are a solution to Maxwell's equations in $\R^3$, and which initially, at time $t=0$, have their support in $\Omega$ and are thus bounded away from the boundary $\Gamma=\partial\Omega$. The objective is to compute the scattered fields $\E\scat=\E\tot-\E\inc$ and $\B\scat=\B\tot-\B\inc$ on a time interval $0\le t \le T$, possibly only at selected space points $\x\in\Omega$, such that the total fields $(\E\tot,\B\tot)$ are
a solution to Maxwell's equations \eqref{maxwell-tot-EB} that satisfies the specified boundary conditions on the boundary $\Gamma$.  As we will construct a numerical method for the computation of the scattered fields, we simply write them as $(\E,\B)=(\E\scat,\B\scat)$.

Throughout the paper, we assume that the physical units are chosen such that
$$
\eps\mu=c^{-2}=1,
$$
which can always be achieved by rescaling time $t \to ct$ or frequency $s \to s/c$.
Furthermore, we rescale the field $\mu\H\to  \H$ where we note that $ \mu \H$ is often referred to as the magnetic field $\Bd$ in the physics literature. 
%

For these rescaled fields, Maxwell's equations then simplify to
\begin{align}\label{maxwell-tot-EB}
\begin{split}
\pt \E\tot-\curl \H\tot &= 0 \quad  \\
\pt \H\tot + \curl \E\tot &=0
\end{split} \quad \text{ in the exterior domain }\Omega,
\end{align}
without the constant factors $\varepsilon$ and $\mu$. 
 This will prove to be a convenient setting for the presentation. 


\subsection{ Generalized impedance boundary conditions}
The time-dependent generalized impedance boundary conditions 
studied here are of the form
\begin{align}\label{GIBC_tot}
\E\tot_T +\Z(\pt)\left(\B\tot\times \normal \right)=0 \quad \text{ on }\Gamma=\partial\Omega,
\end{align}
where $\normal$ denotes the unit surface normal  pointing into the exterior domain $\Om$,  $\E\tot_T$ denotes the tangential component of the total electric field $\E\tot$ on the scattering surface $\Gamma$, and $\Z(\pt)$ is a combined surface differential operator and temporal convolution operator, which in the following is called the  \textit{time-dependent impedance operator}.

 We now give some examples of operators $\Z(\pt)$ from the literature and provide references. 
These boundary operators often contain small quantities, each corresponding to a different physical value. To unify our notation, we will make use of a small parameter $\delta > 0$.

\subsubsection*{ Obstacles with thin coating}
The first boundary condition we are interested in is an approximate model for a perfectly conducting material with a thin coating, as introduced by Engquist \& N\'ed\'elec \cite[equation~(4.9)]{EN93} in the time-harmonic setting.  Transferred to the time domain,  it is given by \eqref{GIBC_tot} with
\begin{align}\label{eq:thin_layer_ord1}
\Z(\pt)&=  \delta  \left(\dfrac{\mu^\delta}{\mu}\pt-\Bigl(\dfrac{\varepsilon^\delta }\varepsilon\Bigr)^{-1} \pt^{-1}\sg\divG \right), 
\end{align}

\noindent where $\delta\ll 1$  is the layer depth and $\varepsilon^\delta,\mu^\delta$ describe the permittivity and permeability inside the thin layer. Here, $\pt^{-1}$ denotes integration in time. This boundary condition is of first-order accuracy in $\delta$. The problem in the time-harmonic setting with a fixed frequency was analysed by Ammari \& N\'ed\'elec \cite{AN96,AN99} using boundary integral equations. In the time-dependent case we are not aware of an analysis of well-posedness or of numerical analysis. Both will be given here.

The boundary condition \eqref{GIBC_tot} with \eqref{eq:thin_layer_ord1} has been extended in several ways, of which we present a small selection in the following.
The second-order boundary condition for thin layers was derived by Haddar \& Joly \cite[Eq.~(95)]{HJ02}.
Its time domain formulation reads
\begin{align}\label{eq:thin_layer_ord2}
\Z(\pt)&= 
\delta \left(
\dfrac{\mu^\delta}{\mu}\pt\bigl( 1+\delta\,(\mathcal{H}-\mathcal{C})\bigr)
-
\Bigl(\dfrac{\varepsilon^\delta }\varepsilon\Bigr)^{-1} \pt^{-1}\sg
\left[ (1-\delta \mathcal{H})\divG\right]\right),
\end{align}
where $\mathcal{H}$ is the mean curvature and $\mathcal{C}$ is the curvature tensor.

In \cite{AH97}, the first-order boundary condition is generalized to a model where the permittivity of the thin coating  is not homogenous but depends on the location on $\Gamma$.


Multiple layers on top of each other is another case of interest, for which effective boundary conditions were recently given in \cite{GLT20}.
The corresponding impedance operator is a linear combination of operators \eqref{eq:thin_layer_ord1} with different permittivities.

\subsubsection*{ Highly conductive obstacles}
A boundary condition for the approximation of scattering from highly conductive obstacles was developed by Haddar, Joly \& Nguyen \cite{HJN08}. The skin effect limits the penetration of the wave to a thin layer near the surface, which then can be asymptotically approximated to create a reduced model.
The authors deduce absorbing impedance boundary conditions for time-harmonic Maxwell's equations of multiple orders. Here we restrict our attention to the first- and second-order boundary conditions.
The impedance operator corresponding to the first-order boundary condition reads in the time domain
\begin{align}\label{eq:gibc_absorbing_ord1}
\quad
\Z(\pt)&=
\delta \,\pt^{1/2},
\end{align}
where $\delta$ is inversely proportional to the high conductivity, and the fractional derivative $\pt^{1/2}$ is the time derivative of convolution with the kernel $(\pi t)^{-1/2}$.
The second-order impedance operator reads
\begin{align}\label{eq:gibc_absorbing_ord2}
\Z(\pt)&=
\delta \,\pt^{1/2}-\delta^2\mu(\mathcal{H}-\mathcal{C}),
\end{align}
where again $\mathcal{H}$ is the mean curvature and $\mathcal{C}$ is the curvature tensor.

\subsection{ Previous related well-posedness analysis and numerical analysis}
Scattering from the above generalized impedance boundary conditions was analysed in the time-harmonic setting for a fixed frequency in the references given above; see in addition Chaulet \cite{Ch16} for well-posedness results in a general framework that partly inspired ours. 
While there is some numerical analysis in the time-harmonic case by  Schmidt \& Hiptmair \cite{SchH15}, we are not aware of any existing numerical analysis of the time-dependent problem as studied here, which requires estimates for the corresponding time-harmonic problem for all frequencies in a complex half-plane in combination with Laplace transform techniques.  Some such frequency-explicit estimates were given for standard electromagnetic boundary integral operators in \cite{BBSV13,ChanMonk2015,KL17}. While conceptually useful, these known estimates do not suffice for a careful numerical analysis of the time-dependent scattering problem with generalized impedance boundary conditions as considered here.

Numerical analysis for time-dependent {\it acoustic} scattering with generalized impedance boundary conditions was recently given in \cite{BLN20}. That paper certainly motivated the present work and it helped for the formulation of the conditions on the impedance operator $\Z(\partial_t)$ and for a first educated guess by analogy as to what a well-posed time-dependent boundary integral equation for the present problem might possibly look like. However, it did not help us in the actual derivation of the boundary integral equation (with its different functional-analytic framework), nor in its well-posedness analysis and numerical analysis, not least because asymptotically sharp frequency-explicit estimates were not fully available in the electromagnetic case and because more refined estimates are needed for the error analysis of the Runge--Kutta based convolution quadrature methods considered here in comparison with the multistep-based convolution quadrature used in \cite{BLN20}.

A fundamental aspect in common with \cite{BLN20} is the use of the coercivity of the Calder\'on operator for frequencies in a complex half-plane, which was proved in \cite{BLS15} for the acoustic case and in \cite{KL17} for the electromagnetic case (with a sign correction in \cite{NKL2020}). A slightly improved version of the electromagnetic coercivity result with explicit constants will be proved here in the course of our well-posedness analysis.

The general approach of approximating wave propagation problems on exterior domains via time-dependent boundary integral equations that are discretized by convolution quadrature in time and boundary elements in space, goes back to \cite{L94}. It has since been often used 
and studied both in the acoustic case, e.g. \cite{LS09,BS09,BLS15,S16,BanjaiRieder}, and in the electromagnetic case, e.g. \cite{ChenMonkWangWeile,BBSV13,ChanMonk2015,KL17}.


\subsection{Contributions of this paper}

This paper gives the first well-posedness analysis and numerical analysis for time-dependent electromagnetic scattering with generalized impedance boundary conditions. This faces the following challenges:
\begin{itemize}
 \item Derive a time-dependent boundary integral equation and 
 prove its well-posedness in appropriate trace spaces; use this to prove well-posedness of the time-dependent scattering problem in appropriate spaces.
 \item Prove stability and convergence (with asymptotically sharp error bounds)
 of the full discretization in space and time.
\end{itemize}

Moreover, in the cases of thin coating and highly conductive materials described above, the estimates should be uniform in the small parameter $\delta$.

While the numerical methods used in this paper, i.e.~boundary elements and convolution quadrature, are well-known methods, it is not obvious {\it a priori} how they are applied to the scattering problem at hand and what their stability and convergence properties are here.
The numerical analysis relies on results proved in the well-posedness analysis.

We will propose, analyse and discretize a system of time-dependent boundary integral equations for the tangential traces of the scattered electric and magnetic fields. Once the tangential traces are known, the scattered fields are obtained at arbitrary points in the exterior domain by well-known time-dependent representation formulas built from the single and double layer electromagnetic potential operators.

In order to prove the well-posedness of the system of time-dependent boundary integral equations and the stability and error bounds of the numerical discretization, a major task is to first prove
	asymptotically sharp frequency-explicit estimates for boundary integral operators 
	for the time-har\-monic Max\-well's equations 
with frequencies in a complex half-plane. These key estimates are derived here systematically via a transmission problem, going beyond related results in the literature. The coercivity of the Calder\'on operator entails the coercivity of the full boundary operator 
including the impedance operator, which yields the well-posedness of the proposed boundary integral equation together with frequency-explicit solution bounds in the natural norms. From these bounds in the Laplace domain, we infer well-posedness of the corresponding time-dependent boundary integral equation in appropriate spaces
and consequently of the time-dependent scattering problem with generalized impedance boundary conditions.

For the time discretization of the system of time-dependent boundary integral equations and the time-dependent representation formulas for the electric and magnetic fields, we use
convolution quadrature based on Radau Runge--Kutta methods, which was first introduced in \cite{LubichOstermann_RKcq} in the context of parabolic problems and was later studied for wave propagation problems in \cite{BLM11}.  Convolution quadrature was  used and analysed for the numerical solution of various exterior Maxwell problems in \cite{BBSV13,ChanMonk2015,ChenMonkWangWeile} and  of an eddy current problem with an impedance boundary condition in \cite{HiptmairLopezFernandezPaganini}.

For space discretization we use boundary elements as described for boundary integral equations related to Maxwell's equations in the monographs by  N\'ed\'elec \cite{Ned01} and Monk \cite{Monk_book}. Here we have chosen Raviart--Thomas elements.

 Using the frequency-explicit bounds from the well-posedness analysis and known approximation properties of the numerical methods, the full discretization of the time-dependent boundary integral equation is proved to be stable and convergent, with explicitly given rates in the case of sufficient regularity. Taking the same
	Runge--Kutta convolution quadrature for discretizing the time-dependent representation formulas,
	we prove full-order  error bounds in time and space in exterior subdomains $\Omega_d\subset \Omega$ with a fixed positive distance $d$ to  the boundary $\Gamma$, both in the $\H(\curl,\Omega_d)$ norm and in the maximum norm on $\Omega_d$, and we prove error bounds of reduced (actually halved) temporal order on the whole exterior domain $\Omega$ in the
 $\H(\curl,\Omega)$ norm, uniformly over bounded time intervals. The error bounds are uniform in the small parameter $\delta$ of the impedance operators in the cases of thin coating and highly conductive materials described above.

\subsection{Outline}

In Section~\ref{section:SettingsGIBC} we introduce the functional-analytic setting of the paper, prove that the above impedance operators fit this framework, and give an appropriate weak formulation of the generalized impedance boundary condition.
Moreover, we introduce  basic notation used throughout the paper.

Section~\ref{section:time harmonic Maxwell} studies the time-harmonic Maxwell's equations with generalized im\-ped\-ance boundary conditions for frequencies in a complex half-plane. The main result here is a well-posedness result for
the time-harmonic scattering problem with a bound that gives an explicit  dependence on the complex frequency (Theorem~\ref{th:time-harmonic-well-posedness}) and behaves well with respect to the small parameter $\delta$ that appears in the time-harmonic impedance operators described above. On the way to proving this result we prove and use boundedness and coercivity results for time-harmonic boundary integral operators, in particular the Calder\'on operator and a related operator that adds the impedance operator.
We formulate the system of boundary integral equations for the tangential traces of the electric and magnetic fields under generalized impedance boundary conditions and prove its well-posedness, showing a bound that is proportional to the square of the absolute value of the complex frequency divided by its real part.
With the tangential traces, the scattered electric and magnetic fields are obtained from the representation formula that involves the single and double layer electromagnetic potential operators.

Section~\ref{section:time dependent Maxwell with GIBC} transfers the results of Section~\ref{section:time harmonic Maxwell} from the Laplace domain
to the time domain, using the polynomial bounds in the frequency together with Laplace transform techniques. We thus obtain well-posedness of the time-dependent electromagnetic scattering problem with generalized im\-ped\-ance boundary conditions (Theorem~\ref{th:time-dependent-well-posedness}) via a system of time-dependent boundary integral equations for the tangential traces of the electric and magnetic fields, which is discretized numerically in the following sections.

Section~\ref{section:time semi-discrete} briefly recapitulates Runge--Kutta based convolution quadratures and their  error bounds as proved in \cite{BLM11}.
 Combining these quadrature error bounds with the time-harmonic well-posedness results of Section~\ref{section:time harmonic Maxwell}, we obtain error bounds for the semi-discretization in time of the system of time-dependent boundary integral equations of Section~\ref{section:time dependent Maxwell with GIBC} and of the scattered time-dependent electric and magnetic fields obtained from the convolution quadrature time discretization of the time-dependent representation formulas.

%
%
In Section~\ref{sec:full} we consider the full discretization of the time-dependent boundary integral equation by Runge--Kutta convolution quadrature in time and Raviart--Thomas boundary elements in space. We
obtain error bounds for the approximate scattered electric and magnetic fields (Theorem~\ref{thm:error-full}). 

In Section~\ref{section:numerics} we present numerical experiments to illustrate our theoretical results and computational aspects.

\section{Framework and analytical background}
\label{section:SettingsGIBC}

We are interested in the solution of the time-dependent Maxwell's equations with generalized impedance boundary conditions in the context of wave scattering. Given an incident wave $\left(\E\inc,\B\inc\right)$, which is a solution to the time-dependent Maxwell's equations on $\R^3$ with initial support in the exterior domain $\Omega$ away from the boundary $\Gamma$, we are interested in computing (possibly in a few selected points $\x$ only) the scattered fields $\E=\E\tot-\E\inc$ and $\B=\B\tot-\B\inc$, which are an outgoing solution to the following initial--boundary value problem of Maxwell's equations:
\begin{alignat}{2}
\pt \E - \curl \B &= 0 \quad &&\text{in} \quad\Omega, \label{MW1}\\
\pt \B + \curl \E &=0 \quad &&\text{in} \quad\Omega,\label{MW2}\\
\E_T +\Z(\pt)\left(\B \times \normal \right) &= \ginc \quad &&\text{on} \quad\Gamma,\label{eq:gibc}
\end{alignat}
where $\E_T=(I-\normal\normal^\top)\E=-(\E\times\normal)\times\normal$ is the tangential component of $\E$ and
\begin{equation}
\label{eq:def g inc}
\ginc = -\bigl(\E\inc_T +\Z(\pt)(\B\inc\times \normal ) \bigr) \quad\text{ on }\Gamma.
\end{equation}
The initial values at $t=0$ are zero in $\Omega$ for both $\E$ and $\B$.


As the problem has finite wave speed $c=1$, the fields $(\E,\B)$ have bounded support at any time, vanishing beyond a distance $ct$ from the boundary at time~$t$. (In contrast to the time-harmonic problem we therefore need not care about asymptotic conditions as $|\x|\to\infty$.)


In this section we describe the functional-analytic framework and show that the above-mentioned examples for $\Z(\pt)$ fit into this general setting. We then give a weak formulation of the boundary condition \eqref{eq:gibc} that is appropriate for our analysis.

\subsection{Tangential trace, trace space $\mXG$ and a further Hilbert space $\mVG\subset\mXG$}
Throughout this paper, we assume that $\Omega$ is the complement of one or several bounded Lipschitz domains in $\R^3$ with boundary surface $\Gamma=\partial\Omega$. For a continuous vector field in the domain, $\v:\overline\Omega\to\C^3$,
we define the {tangential trace}
\begin{align*}
\gaT \v= \v|_\Gamma \times \normal \qquad\text{on }\Gamma,
\end{align*}
where $\normal$ denotes the unit surface normal  pointing into the exterior domain $\Om$.  We note that the tangential component of $\v|_\Gamma$ is $\v_T=(\boldsymb I-\normal\normal^\top)\v|_\Gamma = - (\gaT \v)\times \normal$.

By the version of Green's formula for the $\curl$ operator, we have for sufficiently regular vector fields $\u , \v : \overline\Omega\to\C^3$ that
\begin{equation}\label{eq: Green}
\int_\Omega  \bigl( \curl \u \cdot \v -  \u \cdot \curl \v \bigr)\textrm{d} \x = \int_\Gamma  (\gaT \u \times \normal) \cdot \gaT \v  \,\textrm{d}\sigma,
\end{equation}
where the dot $\cdot$ stands for the Euclidean inner product on $\C^3$, i.e., $\a \cdot \b = \overline{\a}^\top \b$ for $\a,\b\in\C^3$. The right-hand side in
this formula defines a skew-hermitian sesquilinear form on continuous tangential vector fields on the boundary, say $\bphi,\bpsi:\Gamma\to\C^3$, which we write as
\begin{equation}\label{skew}
[\bphi,\bpsi]_\Gamma = \int_\Gamma (\bphi\times\normal)\cdot \bpsi \, \textrm{d}\sigma.
\end{equation}
As it was shown by Alonso \& Valli \cite{AV96} for smooth domains and by Buffa, Costabel \& Sheen \cite{BCS02} for Lipschitz domains
(see also the surveys in \cite[Sect.\,2.2]{BH03} and \cite[Sect.\,5.4]{Ned01}), the trace operator $\gaT$ can be extended to a surjective bounded linear operator from the space that appears naturally for Maxwell's equations,
$
\H(\curl,\Omega) = \{\v \in \bL^2(\Omega)\,:\, \curl \v \in \bL^2(\Omega) \},
$
to the
$$
\text{
	{trace space}: a Hilbert space denoted $\mXG$, with norm $\|\cdot\|_\mXG$.
}
$$
This space is characterized as the tangential subspace of the Sobolev space $\H^{-1/2}(\Gamma)$  with surface divergence in $H^{-1/2}(\Gamma)$ (see the papers cited above for the precise formulation, e.g.~\cite[Section~2.2]{BH03}). It has the property that
the pairing $[ \cdot,\cdot ]_\Gamma$ can be extended to a non-degenerate continuous sesquilinear form on $\mXG\times \mXG$. With this pairing the space $\mXG$ becomes its own dual.

For the treatment of generalized impedance boundary conditions we need a further Hil\-bert space, which is chosen as a dense subspace $\mVG\subset\mXG$  equipped with a (semi-)norm $\abs{\cdot}_{\mVG}$ and the full norm
\begin{equation} \label{full-V-norm}
\norm{\bphi}_{\mVG}^2=\norm{\bphi}_{\mXG}^2+\abs{\bphi}_{\mVG}^2.
\end{equation}
We will choose $\mVG=\mXG\cap \H(\divG,\Gamma)$ with $\H(\divG,\Gamma)= \{\bphi \in \bL^2(\Gamma)\,:\, \divG\, \bphi \in L^2(\Gamma) \}$ for the impedance operators \eqref{eq:thin_layer_ord1} and \eqref{eq:thin_layer_ord2}, and we choose $\mVG=\mXG\cap \bL^2(\Gamma)$ for \eqref{eq:gibc_absorbing_ord1} and
\eqref{eq:gibc_absorbing_ord2}, in all cases with $\abs{\cdot}_{\mVG}$ depending on the small parameter $\delta$.

\subsection{Impedance operator and  recap of  temporal convolution}
\label{subsec:Z}

\noindent 
Let $\Z(s)\colon \mVG \rightarrow \mVG'$, for $\Re s>0$, be an analytic family of bounded linear operators. We assume that $\Z$ is \textit{polynomially bounded}: there exists a real $\kappa$, and for every $\sigma >0$ there exists $M_\sigma <\infty$, such that
\begin{align}\label{eq:pol_bound}
\norm{\Z(s)}_{\mVG'\leftarrow \mVG}&\leq M_\sigma \abs{s}^\kappa, \quad\ \text{ Re } s \ge \sigma>0.
\end{align}
As a key property, we further assume that $\Z$ is of {\it positive type\/}:
for every ${\sigma> \sigma_0\ge 0}$, there exists $c_\sigma>0$ such that
\begin{equation}\label{eq:positive_type}
\Re \langle \bphi, \Z(s)\bphi \rangle_\Gamma \ge c_\sigma\Re  s\, \bigl| s^{-1}\bphi \bigr| _{\mVG}^2
\quad \text{for all } \bphi \in \mVG \text{  and }\text{Re }s\ge \sigma,
\end{equation}
where $\langle\cdot,\cdot\rangle_\Gamma$ denotes the anti-duality between $\mVG$ and $\mVG'$, taken anti-linear in the first argument.
 These conditions are very similar to the conditions imposed on the impedance operator in the acoustic case \cite{BLN20}. 

The bound  \eqref{eq:pol_bound} ensures that $\Z$ is the Laplace transform of a distribution of finite order of differentiation with support on the non-negative real half-line $t \ge 0$. For a function $\g:[0,T]\to \mVG$, which together with its extension by~$0$ to the negative real half-line is sufficiently regular, we use the operational calculus notation
\begin{equation} \label{Heaviside}
\Z(\pt)\g = (\mathcal{L}^{-1}\Z) * \g
\end{equation}
for the temporal convolution of the inverse Laplace transform of $\Z$ with~$\g$. For the multiplication operator $\text{Id}(s)=s$, we have $\text{Id}(\pt)\g = \pt \g$, the time derivative of $\g$.
For two such families of operators $\K(s)$ and $\bL(s)$ mapping into compatible spaces, the associativity of convolution and the product rule of Laplace transforms yield the composition rule
\begin{equation}\label{comp-rule}
\K(\pt)\bL(\pt)\g = (\K\bL)(\pt)\g.
\end{equation}

For a Hilbert space $\V$, we let $\H^r(\R,\V)$ be the Sobolev space of  real   order $r$ of $\V$-valued functions on $\R$,
and on finite intervals $(0, T )$
we denote\footnote{We note that the
subscript 0 in ${H_0^r}$ only refers to the left end-point
of the interval.  }
$$
\H_0^r(0,T;\V) = \{\g|_{(0,T)} \,:\, \g \in \H^r(\R,\V)\ \text{ with }\ \g = 0 \ \text{ on }\ (-\infty,0)\} .
$$
For integer $r\ge 0$, the norm $\| \pt^r \g \|_{\bL^2(0,T;\V)}$ is equivalent to the natural norm on $\H_0^r(0,T;\V)$.
The Plancherel formula yields the following \cite[Lemma 2.1]{L94}:
If $\Z(s)$ is bounded by \eqref{eq:pol_bound} in the half-plane $\text{Re }s > 0$, then $\Z(\pt)$ extends by density to a bounded linear operator $\Z(\pt)$ from $\H^{r+\kappa}_0(0,T;\mVG)$ to $\H^r_0(0,T;\mVG')$ with the bound
\begin{equation}\label{sobolev-bound}
\| \Z(\pt) \|_{ \H^{r}_0(0,T;\mVG') \leftarrow \H^{r+\kappa}_0(0,T;\mVG)} \le e M_{1/ T}
\end{equation}
for arbitrary real $r$. (The bound on the right-hand side arises from the bound $e^{\sigma T} M_\sigma$ on choosing $\sigma=1/T$.)
We note that for any integer $k\ge 0$ and real ${\alpha>\tfrac12}$, we have the continuous embedding $\H^{k+\alpha}_0(0,T;\mVG')\subset \bC^k([0,T];\mVG')$.

The passage from the operators $\Z(s)$, satisfying a polynomial bound \eqref{eq:pol_bound}, to the convolution operators $\Z(\pt)$ and their bound \eqref{sobolev-bound} will be used in the same way also for other operators between different Hilbert spaces in the course of this paper.

\subsection{The impedance operators \eqref{eq:thin_layer_ord1}--\eqref{eq:gibc_absorbing_ord2}}

As the following two lemmas show, the impedance operators listed in the introduction fit into the abstract framework given above.

\begin{lemma}[Thin coating]
\label{lem:EN}
	With the space $\mVG=\mXG\cap \H(\divG,\Gamma)$ and, in \eqref{full-V-norm}, the norm $|\bphi|_\mVG^2= \delta \bigl(\| \bphi \|_{\bL^2(\Gamma)}^2 + \| \divG \bphi \|_{L^2(\Gamma)}^2\bigr)$,
	the transfer operators $\Z(s):\mVG\rightarrow \mVG'$ for $\emph{\text{Re }}s>0$ corresponding to the impedance operators \eqref{eq:thin_layer_ord1} and \eqref{eq:thin_layer_ord2} satisfy the bound \eqref{eq:pol_bound} with $\kappa= 1$ and the positivity condition \eqref{eq:positive_type}, with $M_\sigma$ and $c_\sigma>0$ independent of the small parameter $\delta$.
	In the case of \eqref{eq:thin_layer_ord1}, $\sigma_0=0$ for \eqref{eq:positive_type}.
\end{lemma}

\begin{proof}
	We prove the result only for \eqref{eq:thin_layer_ord1}, as the proof for \eqref{eq:thin_layer_ord2} is a straightforward extension. Moreover, we assume $\eps^\delta$ and $\mu^\delta$ to be positive and restrict our attention, for the ease of presentation, to the transfer operator
	\begin{align*}
	\Z(s)&=\delta\bigl(s - s^{-1}\sg\divG\bigr),
	\end{align*}
	for which the anti-duality between $\mVG$ and $\mVG'$ is to be  understood as follows: for $\bphi,\bpsi\in \mVG$, 
	\begin{equation}\label{Z-EN-weak}
	\langle \bphi, \Z(s)\bpsi \rangle_\Gamma =  \delta s \bigl( \bphi, \bpsi \bigr)_{\Gamma} + \delta s^{-1} \bigl( \divG \bphi, \divG \bpsi \bigr)_{\Gamma},
	\end{equation}
	where the round brackets denote the $L^2$ inner product, taken anti-linear in the first argument. 
	This is bounded as follows, abbreviating $m(|s|)=\max(|s|,|s|^{-1})$:
	\begin{align*}
	|\langle \bphi, \Z(s)\bpsi \rangle_\Gamma| &\le  m(|s|)\, \delta \bigl( \| \bphi \|_{\bL^2(\Gamma)} \,\| \bpsi \|_{\bL^2(\Gamma)}  +
	\| \divG \bphi \|_{L^2(\Gamma)} \, \| \divG \bpsi \|_{L^2(\Gamma)} \bigr)
	\\
	&\le  m(|s|)\, \delta \bigl( \| \bphi \|_{\bL^2(\Gamma)}  + \| \divG \bphi \|_{L^2(\Gamma)}  \bigr)
	\bigl( \| \bpsi \|_{\bL^2(\Gamma)}  + \| \divG \bpsi \|_{L^2(\Gamma)}  \bigr)
	\\
	&\le 2 \, m(|s|) \, |\bphi|_\mVG\, |\bpsi|_\mVG
	\\
	&\le 2  \, m(|s|) \, \|\bphi\|_\mVG\, \|\bpsi\|_\mVG.
	\end{align*}
	This yields \eqref{eq:pol_bound} with $\kappa=1$. On the other hand, taking $\bphi=\bpsi$, we have for $\Re  s\ge \sigma>0$
	\begin{align*}
	\Re \langle \bphi, \Z(s)\bphi \rangle_\Gamma &= \delta\, (\Re  s) \, \| \bphi \|_{\bL^2(\Gamma)}^2
	+ \delta \,\frac{\Re s}{|s|^2}\, \| \divG \bphi \|_{L^2(\Gamma)}^2
	\\
	&\ge \delta (\Re  s)\sigma^2 \| s^{-1} \bphi \|_{\bL^2(\Gamma)}^2 + \delta (\Re  s) \| s^{-1} \divG \bphi \|_{L^2(\Gamma)}^2
	\\
	&\ge \min(\sigma^2,1)\, (\Re  s)\, |s^{-1}\bphi|_\mVG^2,
	\end{align*}
	which yields \eqref{eq:positive_type}.
	\qed
\end{proof}

\begin{lemma}[Highly conductive obstacle] \label{lem:absorb}
	With the space $\mVG=\mXG\cap \bL^2(\Gamma)$ and, in \eqref{full-V-norm}, the norm $|\bphi|_\mVG^2= \delta \| \bphi \|_{\bL^2(\Gamma)}^2 $,
	the transfer operators $\Z(s):{\mVG\rightarrow \mVG'}$ for $\emph{\text{Re }}s>0$ corresponding to the impedance operators \eqref{eq:gibc_absorbing_ord1} and \eqref{eq:gibc_absorbing_ord2} satisfy the bound \eqref{eq:pol_bound} with $\kappa= 1/2$ and the positivity condition \eqref{eq:positive_type}, with $M_\sigma$ and $c_\sigma>0$ independent of the small parameter $\delta$.
	In the case of \eqref{eq:gibc_absorbing_ord1}, $\sigma_0=0$ for \eqref{eq:positive_type}.
\end{lemma}

\begin{proof}
	We prove the result only for \eqref{eq:gibc_absorbing_ord1}, as the proof for \eqref{eq:gibc_absorbing_ord2} is a straightforward extension. Here, the transfer operator is
	\begin{align*}
	\Z(s)&=\delta\, s^{1/2},
	\end{align*}
	for which the anti-duality between $\mVG$ and $\mVG'$ is to be understood,  for $\bphi,\bpsi\in\mVG$,   as
	\begin{equation}\label{Z-absorb-weak}
	\langle \bphi, \Z(s)\bpsi \rangle_\Gamma =  \delta s^{1/2} \bigl( \bphi, \bpsi \bigr)_{\Gamma} .
	\end{equation}
	Here we obtain without ado
	$$
	|\langle \bphi, \Z(s)\bphi \rangle_\Gamma| \le |s|^{1/2} \|\bphi\|_\mVG\, \|\bphi\|_\mVG,
	$$
	which  implies  \eqref{eq:pol_bound} with $\kappa=1/2$,
	and for $\Re  s\ge \sigma>0$ we have
	$$
	\Re \langle \bphi, \Z(s)\bphi \rangle_\Gamma \ge \delta (\Re  s^{1/2}) \, \| \bphi \|_{\bL^2(\Gamma)}^2 \ge
	\sigma^{3/2}(\Re  s) | s^{-1}\bphi |_\mVG^2,
	$$
	which yields \eqref{eq:positive_type}.
	\qed
\end{proof}

\subsection{Weak formulation of the generalized impedance boundary condition}

Formally taking the $\bL^2(\Ga)$ inner product $(\cdot,\cdot)_\Gamma$ of the boundary condition \eqref{eq:gibc} with an arbitrary continuous tangential vector field $\bphi$ on $\Gamma$, we obtain the equation
\begin{equation}
\label{gibc-weak - pre}
(\bphi , \E_T)_\Gamma + ( \bphi , \Z(\pt)\gaT \B )_\Gamma = (\bphi,\g\inc)_\Gamma ,
\end{equation}
which is the starting point for motivating the weak formulation given below.

Noting that for continuous $\E$ we have $\E_T \times \normal= \E\times\normal=\gaT \E$, we find
$$(\bphi, \E_T)_\Gamma = (\bphi\times\normal, \E_T\times\normal)_\Gamma= (\bphi\times\normal,\gaT \E)_\Gamma
= [\bphi,\gaT \E]_\Gamma ,
$$
with the skew-hermitian sesquilinear form \eqref{skew}. Starting from a combined surface differential and temporal convolution operator $\Z(\pt)$ in the strong formulation \eqref{eq:gibc}, we construct the transfer operator $\Z(s):\mVG \to \mVG'$ such that for sufficiently regular $\gaT \B$, the duality coincides with the $\bL^2(\Gamma)$ inner product:
$$
 \langle \bupsilon, \Z(\pt)\gaT \B \rangle_\Gamma  = ( \bupsilon, \Z(\pt)\gaT \B )_\Gamma, \qquad \bupsilon\in\mVG,
$$
as we did for \eqref{eq:thin_layer_ord1}--\eqref{eq:gibc_absorbing_ord2} in \eqref{Z-EN-weak} and \eqref{Z-absorb-weak}.
Similarly,
a regular tangential vector field $\ginc$ defines a functional on $\mVG$ by
$$
 \langle \bupsilon, \ginc \rangle_\Gamma  = ( \bupsilon, \ginc )_\Gamma, \qquad \bupsilon \in \mVG.
$$

Inserting the identities above into  \eqref{gibc-weak - pre} motivates us to study the following \emph{weak formulation of the boundary condition \eqref{eq:gibc}}: the tangential traces of solutions $\E,\B\in \bL^2(0,T;\H(\curl,\Omega))\cap \H^1(0,T;\bL^2(\Omega))$ to the Maxwell's equations in $\Omega$ with zero initial conditions
are to be determined as
$\gaT \E \in \bL^2(0,T;\mXG)$ and 
$\gaT \B \in \H^\kappa_0(0,T;\mVG)$, for $\kappa$ of \eqref{eq:pol_bound}, such that for almost every $t\in (0,T)$,
\begin{equation}\label{gibc-weak}
[\bupsilon,\gaT \E]_\Gamma +  \langle \bupsilon, \Z(\pt)\gaT \B \rangle_\Gamma = \langle\bupsilon,\g\inc\rangle_\Gamma 
\qquad\text{for all $\bupsilon\in\mVG$}.
\end{equation}
This boundary condition relates the tangential traces of $\E$ and $\B$. The terms on the left-hand side are well-defined under the stated regularity requirements on $\gaT \E$ and $\gaT \B$.

In the following two sections we will prove that this initial and boundary value problem is well-posed in the stated Hilbert spaces if  $\g\inc$ has sufficient temporal regularity: $\g\inc\in \H^{3}_0(0,T;\mVG')$ (provided that $\kappa\le 1$, else $\H^{2+\kappa}_0)$. The arguments and intermediate results in the proof of well-posedness will again be used in the stability and error analysis of the numerical methods.

%

\section{Time-harmonic Maxwell's equations}
\label{section:time harmonic Maxwell}

Although the main interest of this work lies on time-domain scattering, it will turn out useful to start with the analysis of the corresponding problem in the Laplace domain, the \emph{time-harmonic Maxwell's equations}  with complex frequencies. 
These equations read, for $s\in\mathbb{C}$ considered here with $\Re s>0$  (or equivalently, $s=-\mathrm{i}\omega$ with the frequency $\omega$ of positive imaginary part), 
see \eqref{maxwell-tot-EB},
\begin{alignat}{2}
\label{TH-MW1}
s \widehat{\E}-\curl \widehat{\B} &= 0
\quad &&\text{in} \ \Omega ,
\\
\label{TH-MW2}
s \widehat{\B} + \curl \widehat{\E} &=0 \quad &&\text{in} \ \Omega .
\end{alignat}
This is complemented with the asymptotic conditions as $|\x|\to\infty$ for an outgoing wave, which are automatically satisfied by the
solutions constructed via the representation formula from the tangential traces on $\Gamma$, as we will do in the following. We will then obtain $\widehat{\E},\widehat{\B} \in \H(\curl,\Omega)$.

 In a series of lemmas in this section, we prove essential estimates for various operators related to the time-harmonic Maxwell's equation. These estimates are explicit in $s$ for $\Re s >0$, in terms of powers of both $|s|$ and $\Re s$. (The precise powers of both are important in Sections~\ref{section:time semi-discrete} and~\ref{sec:full}.) 

We derive a system of boundary integral equations for the tangential traces of $\widehat{\E}$ and $\widehat{\B}$ under time-harmonic generalized impedance boundary conditions and we show well-posedness of the boundary integral equation together with $s$-explicit bounds in appropriate norms. With the representation formulas, we then also show the well-posedness of the time-harmonic scattering problem with generalized impedance boundary conditions, again with $s$-explicit bounds. 

The frequency-explicit bounds will allow us to show the well-posedness of the time-dependent scattering problem in Section~\ref{section:time dependent Maxwell with GIBC}
and to prove higher-order error bounds of the discretization by convolution quadrature and boundary elements in Sections~\ref{section:time semi-discrete} and~\ref{sec:full}.

\subsection{ Recap:  Potential operators and representation formulas}
We  
recall the usual potential operators for the time-harmonic Maxwell's equations; cf.~\cite{BH03,Ned01}.
The {\it fundamental solution} is given by
\begin{align*}
G(s,\x)= \dfrac{e^{-s \abs{\x}}}{4\pi\abs{\x}}, \qquad \Re s>0,\  \x\in\mathbb{R}^3 \setminus \{0\}.
\end{align*}
The electromagnetic \emph{single layer potential} operator $\mathcal{S}(s)$, applied to a regular complex-valued function $\bvar$ and evaluated at $\x\in \mathbb{R}^3 \setminus \Gamma$, is given by 
\begin{align*}
\mathcal{S}(s)\bvar(\x)= -s\int_\Gamma G(s,\x-\y)\bvar(\y)\text{d}\y + s^{-1} \nabla \int_\Gamma G(s,\x-\y) \divG \bvar(\y) \text{d}\y,
\end{align*}
and the electromagnetic \emph{double layer potential} operator $\mathcal{D}(s)$ is
given by
\begin{align*}
\mathcal{D}(s)\bvar(\x) = \curl\int_\Gamma G(s,\x-\y)\bvar(\y)\text{d}\y.
\end{align*}
The potential operators satisfy the relations
\begin{equation}\label{SD-MW}
s \mathcal{S}(s) - \curl \circ\, \mathcal{D}(s) = 0, \qquad s \mathcal{D}(s) + \curl \circ\, \mathcal{S}(s) =0.
\end{equation}
This implies that for any regular function $\bvar$, the fields $\widehat \E = \mathcal{S}(s)\bvar$ and $\widehat \B = \mathcal{D}(s)\bvar$ are a solution to the time-harmonic Maxwell's equations \eqref{TH-MW1}--\eqref{TH-MW2} on $\mathbb{R}^3 \setminus \Gamma$ (recall $\eps\mu=1$). Likewise, this  also holds true for the fields $\widehat \E = \mathcal{D}(s)\bvar$ and $\widehat \B = -\mathcal{S}(s)\bvar$.

In our problem setting only the exterior domain $\Omega$ matters. As a theoretical tool, however, it will be useful to analyse transmission problems on $\mathbb{R}^3\setminus \Gamma$. We introduce some standard notation designed to simplify the description of such problems.

In the context of transmission problems, we denote the interior of the bounded scatterer by $\Omega^-$ and the exterior domain by $\Omega^+$ (elsewhere in this paper denoted by $\Omega$), such that $\mathbb{R}^3$ is decomposed into $\mathbb{R}^3=\Omega^-\,\dot\cup\,\Gamma \,\dot\cup\, \Omega^+$. Furthermore, $\gaT^-$ and $\gaT^+$ denote the tangential traces on $\Omega^-$ and $\Omega^+$, respectively. We denote \emph{jumps} and \emph{averages}
by
\begin{align*}
\jmp{\gaT}=\gaT^+-\gaT^-,\quad\quad
\avg{\gaT}=\tfrac12\left(\gaT^++\gaT^-\right).
\end{align*}
The sign convention for the jumps has been chosen to coincide with that of~\cite{BH03}.
A fundamental role is played by the {\it jump relations} of the potential operators:
\begin{equation}\label{jump-rel}
\jmp{\gaT}\circ \mathcal{S}(s) =0, \qquad \jmp{\gaT}\circ \mathcal{D}(s) = -  \Id  .
\end{equation}
As a direct consequence of \eqref{SD-MW} and \eqref{jump-rel}, for any given boundary densities $(\wvarphi,\wpsi)$ (regular in a dense subspace of $\mXG\times \mXG$), the electric and magnetic fields defined by\footnote{We write $(\wvarphi,\wpsi)$ when these functions appear as boundary densities defining fields $(\widehat{\E},\widehat{\B})$ as in \eqref{eq:time-harmonic-kirchhoff-E3}--\eqref{eq:time-harmonic-kirchhoff-H3}, where the hats recall that these variables correspond to Laplace transforms of time-dependent functions, which will be studied in the next section. On the other hand, we omit the hats for generic functions to which potential operators or boundary operators are applied.}
\begin{align}
	\widehat{\E}&=-\mathcal{S}(s)\wvarphi + \mathcal{D}(s)\wpsi,\label{eq:time-harmonic-kirchhoff-E3}
	\\
	\widehat{\B}&= - \mathcal{D}(s)\wvarphi -\mathcal{S}(s)\,\wpsi, \label{eq:time-harmonic-kirchhoff-H3}
\end{align}
are a solution to the transmission problem 
	\begin{alignat}{3}
	&s \widehat{\E}-\curl \widehat{\B} &&=0 
	\quad\quad\quad &&\text{in} \quad \mathbb{R}^3\setminus\Gamma ,\label{eq:transmis-1}
	\\
	&s \widehat{\B} + \curl \widehat{\E} &&=0 \quad &&\text{in} \quad \mathbb{R}^3\setminus\Gamma ,\label{eq:transmis-2} \\
	&\ \,\jmp{\gaT}\widehat{\B}=\wvarphi\,,&&\label{eq:transmis-3} \\
	&-\jmp{\gaT} \widehat \E=\wpsi \,.&&\label{eq:transmis-4}
	\end{alignat}
So far in this section, we recalled well-known identities and our presentation was restricted to regular boundary densities.

\subsection{ Frequency-explicit bounds for solutions of the transmission problem}
The following lemma shows that the linear map
$(\wvarphi,\wpsi)\mapsto (\widehat{\E},\widehat{\B})$ extends by density to a bounded linear operator from $\mXG\times \mXG$ to
$\H(\curl,\Omega)\times \H(\curl,\Omega)$, and it gives an $s$-explicit bound; cf.~\cite[Lemma~6.4]{ChanMonk2015} for a related, yet more complicated result.
\begin{lemma}\label{lem:transmission}
For  $\Re s >0$, the solution $(\widehat \E,\widehat \B)$ of the transmission problem \eqref{eq:transmis-1}--\eqref{eq:transmis-4} defined by \eqref{eq:time-harmonic-kirchhoff-E3}--\eqref{eq:time-harmonic-kirchhoff-H3}
is bounded by
$$
\left\|\begin{pmatrix} \widehat{\E} \\ \widehat{\B} \end{pmatrix} \right\|_{\H(\curl,\mathbb{R}^3\setminus \Gamma)^2}
\le  \Ctrace \, \dfrac{\abs{s}^{2}+1}{\Re s}\,
\left\|\begin{pmatrix} \wvarphi \\ \wpsi \end{pmatrix} \right\|_{\mXG^2},
$$
where $\Ctrace=\| \{\gaT\} \|_{\mXG \leftarrow \H(\curl,\mathbb{R}^3\setminus\Gamma)}$.
\end{lemma}

\begin{proof}  We start from Green's formula \eqref{eq: Green}--\eqref{skew} on the exterior and interior domain and note that 
for solutions of the time-harmonic Maxwell's equations \eqref{TH-MW1}--\eqref{TH-MW2}, Green's formula  reduces to
\begin{align}
\nonumber
\pm\left[\gaT^\pm \widehat \B, \gaT^\pm \widehat \E\right]_\Gamma
&=
\int_{\Omega^\pm} \bigl(   \curl \widehat\B\cdot \widehat \E - \widehat \B \cdot\curl \widehat \E  \bigr)\,\textrm{d} x
\\
\label{Green-EH}
&=
\int_{\Omega^\pm}  \bigl(\bar s  \big| \widehat \E \big|^2+  s  \big| \widehat \B \big|^2\bigr)\, \textrm{d} x.
\end{align}
The conjugation of $s$ in the first summand stems from the convention that $\cdot$ denotes the inner product $\a \cdot \b = \overline{\a}^\top \b$ on $\mathbb{C}^3$. Summing up,  we obtain 
\begin{equation}\label{Green-EH-trans}
I:= \int_{\mathbb{R}^3\setminus \Gamma}  \bar s \big| \widehat \E \big|^2+  s  \big|  \widehat \B \big|^2 \textrm{d} x
=  \left[\gaT^+ \widehat \B, \gaT^+ \widehat \E\right]_\Gamma
- \left[\gaT^- \widehat \B,  \gaT^- \widehat \E\right]_\Gamma .
\end{equation}
 On inserting \eqref{eq:transmis-1} and \eqref{eq:transmis-2} for $\widehat \E$ and $\widehat \B$ into $\theta$ times the integrand, where $0<\theta<1$ is arbitrary,
the left-hand side is rewritten as
	\begin{align*}
	I &=
	 \int_{\mathbb{R}^3\setminus \Gamma} \Big( (1-\theta) \bar s \big|\widehat{\E}\big|^2
	+\theta s \big|s^{-1}\curl \widehat{\E}\big|^2
	\\
	& \qquad \qquad \quad +
	(1-\theta) s  \big| \widehat{\B}\big|^2
	+\theta \bar s  \big| s^{-1}  \curl \widehat{\B}\big|^2  \Big)
	\mathrm{d}x .
	\end{align*}
Choosing $\theta$ such that $1-\theta=\theta |s|^{-2}$, i.e. $\theta=1/(1+|s|^{-2})$, and taking the real part then gives
\begin{equation}\label{ReI-1}
\Re I =
\frac{\Re s}{|s^2|+1} \Bigl( \| \widehat \E \|_{\H(\curl,\R^3\setminus\Gamma)}^2 + \|  \widehat \B \|_{\H(\curl,\R^3\setminus\Gamma)}^2 \Bigr).
\end{equation}
On the other hand, by \eqref{Green-EH-trans} we also have
$$
\Re I =
\Re \Bigl( \left[\gaT^+ \widehat \B, \gaT^+ \widehat \E\right]_\Gamma
- \left[\gaT^- \widehat \B,  \gaT^- \widehat \E\right]_\Gamma \Bigr).
$$
Rewriting the right-hand side in terms of jumps and averages and using the transmission conditions \eqref{eq:transmis-3}--\eqref{eq:transmis-4}, we obtain
\begin{align} \label{ReI-2}
\Re I
&= \Re \Bigl( \left[ \jmp{\gaT }\widehat \B, \avg{\gaT}\widehat{\E} \right]_\Gamma  + \left[ -\jmp{\gaT} \widehat \E, \avg{\gaT}\widehat{\B}\right]_\Gamma \Bigr)
\\ \nonumber
&=
\Re \Bigl( \left[ \wvarphi, \avg{\gaT}\widehat{\E} \right]_\Gamma
+ \left[ \wpsi, \avg{\gaT}\widehat{\B}\right]_\Gamma \Bigr).
\end{align}
We now recall that $\mXG$ is its own dual with the duality pairing $[\cdot,\cdot]_\Gamma$ and we use the Cauchy--Schwarz inequality on $\R^2$ 
to estimate
\begin{align*}
\Re I & \le \| \wvarphi \|_{\mXG}  \, \| \avg{\gaT}\widehat{\E} \|_{\mXG} +
\| \wpsi \|_{\mXG}  \, \| \avg{\gaT}\widehat{\B} \|_{\mXG}
\\
&\le \Bigl( \| \wvarphi \|_{\mXG} ^2 + \| \wpsi \|_{\mXG} ^2 \Bigr)^{1/2} 
\Bigl( \| \avg{\gaT}\widehat{\E} \|_{\mXG}^2 + \|\avg{\gaT}  \widehat{\B} \|_{\mXG}^2 \Bigr)^{1/2}.
\end{align*}
The right-hand side is finite because it is known from \cite{BH03} that $\widehat \E$ and $\widehat \B$ are in the local Sobolev space
$\H_{\mathrm{loc}}(\curl,\R^3\setminus\Gamma)$ and moreover, $\avg{\gaT}$ is a bounded operator from $\H(\curl,\Omega_R)$ onto $\mXG$, where $\Omega_R$ is a ball of sufficiently large radius $R$ that contains $\Gamma$. So we find that $\Re I$ has a finite bound, and by \eqref{ReI-1}, $\widehat \E$ and $\widehat \B$ are therefore in $\H(\curl,\R^3\setminus\Gamma)$. We then use the bound $C_\Gamma$ of
$\avg{\gaT}: \H(\curl,\R^3\setminus\Gamma) \to \mXG$ to conclude
$$
\Re I \le C_\Gamma \Bigl( \| \wvarphi \|_{\mXG} ^2 + \| \wpsi \|_{\mXG} ^2 \Bigr)^{1/2}
\Bigl( \| \widehat \E \|_{\H(\curl,\R^3\setminus\Gamma)}^2 + \|  \widehat \B \|_{\H(\curl,\R^3\setminus\Gamma)}^2 \Bigr)^{1/2}.
$$
In view of \eqref{ReI-1}, this yields the stated result.
\qed
\end{proof}

On setting $\wpsi=0$ in Lemma~\ref{lem:transmission}, we immediately obtain the following corollary.

\begin{lemma}\label{lem:potential-SD-bound}
For $\Re s>0$, the  single and double layer potential operators $\mathcal S(s)$ and $\mathcal D(s)$ extend by density to bounded linear operators from $\mXG$ to ${\H(\curl,\mathbb{R}^3\setminus \Gamma)}$, which are bounded by 
	\begin{align*}
	&\| \mathcal S(s) \|_{\H(\curl,\mathbb{R}^3\setminus \Gamma) \leftarrow \mXG} \le \Ctrace \dfrac{\abs{s}^{2}+1}{\Re s},
	\\[1mm]
	&
	\| \mathcal D(s) \|_{\H(\curl,\mathbb{R}^3\setminus \Gamma) \leftarrow \mXG} \le \Ctrace \dfrac{\abs{s}^{2}+1}{\Re s},
	\end{align*}
where again $\Ctrace=\| \{\gaT\} \|_{\mXG \leftarrow \H(\curl,\mathbb{R}^3\setminus\Gamma)}$.
\end{lemma}

We return to the transmission problem \eqref{eq:transmis-1}--\eqref{eq:transmis-4}.
Electromagnetic scattered fields $\widehat{\E},\widehat{\B}$ that solve \eqref{TH-MW1}--\eqref{TH-MW2} in the exterior domain $\Omega=\Omega^+$ are extended by zero into the interior, so that the jumps are just the exterior tangential traces in \eqref{eq:transmis-3}--\eqref{eq:transmis-4}, as are the averages up to the factor $1/2$.
The scattered fields are then recovered from their tangential traces by the representation formulas 
\begin{align}
\label{eq:time-harmonic-kirchhoff-E}
\widehat{\E} &= - \, \mathcal{S}(s)\bigl( \gaT \widehat{\B}\bigr) + \mathcal{D}(s)\bigl(-\gaT\widehat{\E}\bigr)\quad \text{ in } \Omega , \\
\label{eq:time-harmonic-kirchhoff-H}
\widehat{\B} &= - \mathcal{D}(s) \bigl( \gaT \widehat{\B}\bigr) \,- \, \mathcal{S}(s) \bigl(-\gaT\widehat{\E}\bigr) \quad \text{ in } \Omega .
\end{align}
Our analytical as well as numerical approach will consist in determining the tangential traces from boundary integral equations that incorporate the generalized impedance boundary conditions, and then obtain the electromagnetic fields from the above representation formulas (or their time-domain analogues).

In this situation the bound of Lemma~\ref{lem:transmission} improves as follows.

\begin{lemma} \label{lem:transmission-0}
In the situation of Lemma~\ref{lem:transmission}, assume further that the interior tangential traces of $\widehat{\E}$ and $\widehat{\H}$ are identically $0$, which implies $ \gaT \widehat{\B} = \wvarphi$ and  $-\gaT \widehat{\E} = \wpsi$. Then, the bound of Lemma~\ref{lem:transmission} improves to
$$
\left\|\begin{pmatrix} \widehat{\E} \\ \widehat{\B} \end{pmatrix} \right\|_{\H(\curl,\Omega)^2}
\le  \left( \dfrac{\abs{s}^{2}+1}{2\Re s}\right)^{1/2}
\left\|\begin{pmatrix} \wvarphi \\ \wpsi \end{pmatrix} \right\|_{\mXG^2}.
$$
Furthermore, we have the $L^2$ bound
$$
\left\|\begin{pmatrix} \widehat{\E} \\ \widehat{\B} \end{pmatrix} \right\|_{\bL^2(\Omega)^2}
\le  \left( \dfrac{1}{2\Re s}\right)^{1/2}
\left\|\begin{pmatrix} \wvarphi \\ \wpsi \end{pmatrix} \right\|_{\mXG^2}.
$$
\end{lemma}

\begin{proof} The proof of the $\H(\curl,\Omega)$ bound is identical to that of Lemma~\ref{lem:transmission} down to \eqref{ReI-2}, which now implies the bound
$
\Re I \le \tfrac12\,\bigl(\| \wvarphi \|_{\mXG} ^2 + \| \wpsi \|_{\mXG} ^2\bigr)
$
and yields the stated result. The proof of the $L^2$ bound is even simpler, working directly with \eqref{Green-EH-trans} instead of
\eqref{ReI-1}.
\qed
\end{proof}


\subsection{Time-harmonic boundary operators and the Calder\'{o}n operator}

The electromagnetic \emph{single and double layer boundary operators} are the operators from $\mXG$ to $\mXG$ defined as
$$
\V(s)=\avg{\gaT}\circ \mathcal{S}(s), \qquad \K(s)=\avg{\gaT}\circ \mathcal{D}(s).
$$
%
%
We define the \emph{Calder\'{o}n operator} as introduced in \cite{KL17} (with a sign corrected in \cite{NKL2020}):
\begin{equation} \label{def-B}
\Cald(s)=\begin{pmatrix}
-\V(s) & \K(s) \\
-\K(s) & -\V(s)
\end{pmatrix}
= \avg{\gaT}\circ
\begin{pmatrix}
-\mathcal{S}(s) & \mathcal{D}(s) \\
-\mathcal{D}(s) & -\mathcal{S}(s)
\end{pmatrix},
\end{equation}
where we note that the right-most block operator is the one appearing in the representation formula \eqref{eq:time-harmonic-kirchhoff-E3}--\eqref{eq:time-harmonic-kirchhoff-H3}.
Let $\widehat \E,\widehat \B $
be Maxwell solutions that are given by this representation formula. Then, by construction
the Calder\'on operator maps jumps to averages (see \eqref{eq:transmis-1}--\eqref{eq:transmis-4}):
\begin{align}\label{eq:calderon-jump}
\Cald(s)\begin{pmatrix}\jmp{\gaT }\widehat \B\\ -\jmp{\gaT} \widehat \E \end{pmatrix}=
\begin{pmatrix}
\avg{\gaT}\widehat{\E} \\ \avg{\gaT}\widehat{\B}
\end{pmatrix} .
\end{align}
The following bound of $\Cald(s)$ follows immediately from \eqref{def-B} and Lemma~\ref{lem:transmission}.
This bound improves on existing time-harmonic $s$-explicit bounds of the boundary operators; see  the $O(|s|^2)$ bounds in  \cite[Theorem 4.4]{BBSV13} and~\cite[Lemma~2.3]{KL17}.
\begin{lemma}
	\label{lemma:B-boundedness}
	For $\Re s>0$, the Calder\'on operator $\Cald(s):\mXG^2 \to \mXG^2$ is bounded by
		\begin{equation}
	\label{calderon-strong-bound}
	\norm{\Cald(s)}_{\mXG^2\leftarrow\mXG^2}\le \Ctrace^2 \,\dfrac{\abs{s}^2+1}{\Re s} ,
	\end{equation}
	where again $\Ctrace=\| \{\gaT\} \|_{\mXG \leftarrow \H(\curl,\mathbb{R}^3\setminus\Gamma)}$.
	The same bound also holds for $\norm{\V(s)}_{\mXG\leftarrow \mXG}+\norm{\K(s)}_{\mXG\leftarrow \mXG}$.
\end{lemma}

%

We extend the skew-hermitian pairing $[\cdot,\cdot]_\Gamma$ from $\mXG\times \mXG$ to $\mXG^2\times \mXG^2$ in the obvious way:
$$
\left[ \begin{pmatrix} \bvar \\ \bpsi \end{pmatrix} , \begin{pmatrix} \bupsilon \\ \bxi \end{pmatrix} \right]_\Gamma =
[ \bvar,\bupsilon]_\Gamma + [ \bpsi,\bxi]_\Gamma.
$$
The Calder\'on operator $\Cald(s)$ is coercive with respect to the pairing $[\cdot,\cdot]_\Gamma$, as was shown in \cite[Lemma~3.1]{KL17}.  Here we give an improved formulation of this key lemma with an $s$-explicit bound  and we provide a restructured proof.

\begin{lemma}[Coercivity of the Calder\'on operator]
	\label{lem:B-coercivity}
	For $\Re s>0$,  we have the coercivity
	\begin{equation}
	\label{eq:coercivity Calderon}
	\textnormal{Re} \left[
	\begin{pmatrix} \bvar\\ \bpsi \end{pmatrix}
	,\Cald(s)
	\begin{pmatrix} \bvar\\ \bpsi \end{pmatrix}
	\right]_\Gamma
	\ge \frac 1{c_\Gamma^{2}} \, \frac{\Re s}{|s|^{2}+1}\, \left(\bigl\|{\bvar}\bigr\|^2_{\mXG}+\bigl\|\bpsi\bigr\|^2_{\mXG}\right)
	\end{equation}
	for all $(\bvar,\bpsi) \in \mXG^2 $. Here,  $\ctrace=\| [\gaT] \|_{\mXG \leftarrow \H(\curl,\mathbb{R}^3\setminus\Gamma)}$.
\end{lemma}

\begin{proof}
	Let $(\wvarphi,\wpsi)\in \mXG^2$ be arbitrary and $\widehat{\E},\widehat{\B} \in \H(\curl,\mathbb{R}^3\setminus \Gamma)$ be the solutions to the associated transmission problem of Lemma~\ref{lem:transmission}.  We then have
	\begin{align*}
&\norm{\begin{pmatrix}
		\wvarphi\\ \wpsi
		\end{pmatrix}
	}^2_{\mXG\times\mXG}
	 =
	\norm{\begin{pmatrix}\jmp{\gaT }\widehat \B
	\\ -\jmp{\gaT} \widehat \E\end{pmatrix}
	}^2_{\mXG\times\mXG}
	& \quad\text{by~\eqref{eq:transmis-3}--\eqref{eq:transmis-4}}
	\\
	& \le \ctrace^2 \left(  \big\| \widehat{\B} \big\|_{\H(\curl,\mathbb{R}^3\setminus \Gamma)}^2
	+ \big\| \widehat{\E} \big\|_{\H(\curl,\mathbb{R}^3\setminus \Gamma)}^2
        \right)
        &\quad\text{by def.~of $c_\Ga$}
        \\
        &= c_\Gamma^2 \,\dfrac{\abs{s}^2+1}{\Re s}	\,
	\Re \left[
        \begin{pmatrix}\jmp{\gaT }\widehat \B \\ -\jmp{\gaT} \widehat \E \end{pmatrix},
        \begin{pmatrix}\avg{\gaT}\widehat{\E} \\ \avg{\gaT}\widehat{\B} \end{pmatrix}
        \right]_\Gamma
        &\quad \text{by \eqref{ReI-1}--\eqref{ReI-2}}
         \\
         &=  c_\Gamma^2 \,\dfrac{\abs{s}^2+1}{\Re s}	\,
	\Re \left[
	\begin{pmatrix}\jmp{\gaT }\widehat \B\\ -\jmp{\gaT} \widehat \E \end{pmatrix},
	\Cald(s)\!\begin{pmatrix}\jmp{\gaT }\widehat \B\\ -\jmp{\gaT} \widehat \E \end{pmatrix}\right]_\Gamma
	&\quad\text{by \eqref{eq:calderon-jump}}
	\\
	&=  c_\Gamma^2 \,\dfrac{\abs{s}^2+1}{\Re s}	\,
	\Re \left[
	\begin{pmatrix} \wvarphi\\ \wpsi \end{pmatrix},
	\Cald(s)\! \begin{pmatrix} \wvarphi\\ \wpsi \end{pmatrix}\right]_\Gamma
	&\quad\text{by~\eqref{eq:transmis-3}--\eqref{eq:transmis-4}}.
	\end{align*}
	
This yields the result.
\qed
\end{proof}

\subsection{Boundary integral equation for tangential traces under time-harmonic generalized impedance boundary conditions}
%
We now derive a well-posed boundary integral equation of the time-harmonic Maxwell's equations \eqref{TH-MW1}--\eqref{TH-MW2}  for $\Re s>0$ with the weak formulation of the generalized impedance boundary condition \eqref{gibc-weak},
\begin{equation}\label{gibc-weak-harmonic}
[\bupsilon,\gaT \widehat{\E}]_\Gamma + \langle \bupsilon, \Z(s)\gaT \widehat{\B} \rangle_\Gamma = \langle \bupsilon,\widehat{\g}\inc\rangle_\Gamma
\qquad\text{for all $\bupsilon\in\mVG$},
\end{equation}
where the transfer operator $\Z(s)$ satisfies \eqref{eq:pol_bound}--\eqref{eq:positive_type}, and $\wginc\in \mVG'$ is arbitrary.

We start with the observation that any solution of the time-harmonic Maxwell's  equations on the exterior domain $\Omega=\Omega^+$, trivially extended by zero into the bounded interior
$\Omega^-$, solves an associated transmission problem as in Lemma~\ref{lem:transmission}.
As the inner traces of the extended fields vanish by construction, the jumps and the averages reduce to the outer traces and the representation formulas can be evaluated by the boundary data, as in \eqref{eq:time-harmonic-kirchhoff-E}--\eqref{eq:time-harmonic-kirchhoff-H}.

Therefore, the relation \eqref{eq:calderon-jump} of the Calder\'on operator then reads
\begin{align}\label{eq:cald-identity}
\Cald(s)
\begin{pmatrix}
\gaT \widehat{\B} \\
-\gaT \widehat{\E}
\end{pmatrix}
= \dfrac{1}{2}
\begin{pmatrix}
\gaT\widehat{\E} \\
\gaT \widehat{\B}
\end{pmatrix} .
\end{align}
In analogy to \cite{BanjaiRieder,BanjaiLubich2019,BLN20} in the acoustic case, where a skew-symmetric operator is added to the acoustic Calder\'on operator, we rewrite this identity by adding a {\it symmetric} block operator and arrive at
\begin{align}\label{eq:Bimp_s}
\begin{split}
\Bimp(s)
\begin{pmatrix}
\gaT \widehat{\B} \\
-\gaT \widehat{\E}
\end{pmatrix}
=
\begin{pmatrix}
\gaT\widehat{\E} \\
0
\end{pmatrix},
\quad \quad
\Bimp(s) = \Cald(s)+
\begin{pmatrix}
0 & -\tfrac{1}{2}\Id \\ 
-\tfrac{1}{2}\Id  & 0
\end{pmatrix}
\end{split}.
\end{align}
We introduce the boundary densities
\begin{equation}\label{phi-psi}
	{\wvarphi=\gaT\widehat{\B}}, \qquad {\wpsi=-\gaT \widehat{\E}} ,
\end{equation}
and test both sides with $(\bupsilon,\bxi)\in \mVG \times\mXG$. This yields
\begin{align*}
\left[
\begin{pmatrix}
\bupsilon \\ \bxi
\end{pmatrix},
\Bimp(s)
\begin{pmatrix}
\wvarphi \\
\wpsi
\end{pmatrix}
\right]_\Gamma
=\left[
\bupsilon,\gaT\widehat{\E}\right]_\Gamma.
\end{align*}
Inserting the boundary condition \eqref{gibc-weak-harmonic} on the right-hand side then leads
to the weak formulation of the boundary integral equation that will be studied here.

\bigskip\noindent
{\bf Boundary integral equation:}
{\it For $\Re s>0$ and given $\wginc\in \mVG'$, find $(\wvarphi, \wpsi)\in \mVG\times\mXG$ such that, for all $(\bupsilon,\bxi) \in \mVG\times\mXG$,}
\begin{align}
\label{bie-weak}
\left[
\begin{pmatrix}
\bupsilon \\ \bxi
\end{pmatrix},
\Bimp(s)
\begin{pmatrix}
\wvarphi \\
\wpsi
\end{pmatrix}
\right]_\Gamma+
\langle \bupsilon, \Z(s) \wvarphi \rangle_\Gamma = \langle \bupsilon,\widehat{\g}\inc\rangle_\Gamma.
\end{align}

\bigskip
We introduce the family of operators ${\A(s):\mVG\times\mXG\rightarrow \mVG'\times\mXG'}$ that is defined by the left-hand side above, i.e.,
for all $(\bvar,\bpsi)$ and $(\bupsilon,\bxi)\in \mVG\times\mXG$,
\begin{align}\label{Asoperator}
\left\langle
\begin{pmatrix}
\bupsilon \\ \bxi
\end{pmatrix},
\A(s)\begin{pmatrix}\bvar \\ \bpsi\end{pmatrix}\right\rangle_\Gamma=
\left[
\begin{pmatrix}
\bupsilon \\ \bxi
\end{pmatrix},
\Bimp(s)
\begin{pmatrix}
\bvar \\
\bpsi
\end{pmatrix}
\right]_\Gamma+
\langle \bupsilon, \Z(s)\bvar\rangle_\Gamma,
\end{align}
where $\langle \cdot,\cdot\rangle_\Gamma$ denotes the anti-duality between  $\mVG\times\mXG$ and $\mVG'\times\mXG'$ on the left-hand side, and between $\mVG$ and $\mVG'$ on the right-hand side. The boundary integral equation \eqref{bie-weak} then reads more compactly as follows: find $(\wvarphi, \wpsi)\in \mVG\times\mXG$ such that
\begin{equation}\label{bie-weak-A}
\left\langle
\begin{pmatrix}
\bupsilon \\ \bxi
\end{pmatrix},
\A(s)\begin{pmatrix}\wvarphi \\ \wpsi\end{pmatrix}\right\rangle_\Gamma=
\langle\bupsilon,\widehat{\g}\inc\rangle_\Gamma
\qquad \text{for all }\ (\bupsilon,\bxi) \in \mVG\times\mXG.
\end{equation}

We write \eqref{bie-weak-A} even more compactly as
\begin{equation}\label{bie-A}
\A(s) \begin{pmatrix}
\wvarphi \\
\wpsi
\end{pmatrix} = \begin{pmatrix}
\widehat{\g}\inc \\
\mathbf{0}
\end{pmatrix}.
\end{equation}
%

The boundary integral operator $\A(s)$ defined by \eqref{Asoperator} inherits the bounds and positivity properties of the Calder\'{o}n operator  $\Cald(s)$ and the impedance operator $\Z(s)$, respectively, from Lemma~\ref{lemma:B-boundedness}--\ref{lem:B-coercivity} and \eqref{eq:pol_bound}--\eqref{eq:positive_type}, as the following two lemmas state.

\begin{lemma}
	\label{lem:A-boundedness}
	The operators $\A(s):\mVG\times\mXG \rightarrow \mVG'\times \mXG'$ defined by \eqref{Asoperator} form an analytic family of bounded linear operators that satisfy the bound, for $\Re s \ge \sigma > 0$,
	\begin{align*}
	\norm{\A(s)}_{\mVG'\times\mXG' \leftarrow \mVG\times \mXG}\le C_\sigma \dfrac{\abs{s}^2}{\Re s}.
	\end{align*}
	The constant $C_\sigma$ only depends polynomially on $\sigma^{-1}$ and on the boundary $\Gamma$ via the norm of the tangential trace operator.
\end{lemma}
\begin{proof}
	The bound \eqref{calderon-strong-bound}  of the Calder\'on operator $\Cald(s)$ given in Lemma \ref{lem:B-coercivity} and the polynomial bound \eqref{eq:pol_bound} of the impedance operator $\Z(s)$ yield estimates on all the terms appearing on the right-hand side of \eqref{Asoperator} with the exceptions of the identity operators $\Id_{\mVG' \leftarrow \mXG}$ and $\Id_{\mXG' \leftarrow \mVG}$ occurring in $\Bimp$, which 
	are bounded in view of the continuous embeddings ${\mVG \subset \mXG = \mXG' \subset \mVG'}$.
	\qed
\end{proof}

\begin{lemma}
	\label{lem:A-coercivity}
	The operator family $\A(s)$ has the following coercivity property:
	For every $ \sigma>\sigma_0$ (with $\sigma_0\ge 0$ of \eqref{eq:positive_type}), there exists a constant $c_\sigma > 0$ such that
	for $\Re s \ge \sigma$,
	\begin{align*}
	\Re \left\langle
	\begin{pmatrix} \bvar\\ \bpsi \end{pmatrix}
	,\A(s)
	\begin{pmatrix} \bvar\\ \bpsi \end{pmatrix}
	\right\rangle_\Gamma\ge  c_{\sigma} \, \frac{\Re s}{|s|^2}  \left(\bigl\|{\bvar}\bigr\|^2_{\mVG}+\bigl\|\bpsi\bigr\|^2_{\mXG}\right) 
	\end{align*}
	for all $ (\bvar,\bpsi) \in \mVG\times \mXG$. The constant $c_\sigma$ only depends polynomially on $\sigma^{-1}$ and on the boundary $\Gamma$ via the norm of the tangential trace operator.
\end{lemma}
\begin{proof}
	The operator $\Bimp(s)$ has the same coercivity property as $\Cald(s)$ because with the skew-hermitian pairing we have
	\begin{align*}
		\Re
	\left[
	\begin{pmatrix}
	\bvar \\ \bpsi
	\end{pmatrix},
	\begin{pmatrix}
	0 & \Id \\ 
	\Id & 0
	\end{pmatrix}
	\begin{pmatrix}
	\bvar \\
	\bpsi
	\end{pmatrix}
	\right]_\Gamma
	=
		\Re\bigl(
	\left[\bvar,\bpsi \right]_\Gamma+\left[\bpsi,\bvar \right]_\Gamma
		\bigr)
	=0.
	\end{align*}
	Combining the coercivity of the Calder\'on operator, as stated in Lemma~\ref{lem:B-coercivity}, and the positivity condition \eqref{eq:positive_type} on $\Z(s)$ then yield
	\begin{align*}
	\Re
	\left\langle
	\begin{pmatrix}
	\bvar \\	\bpsi
	\end{pmatrix},
	\A(s)\begin{pmatrix}\bvar \\ \bpsi\end{pmatrix}\right\rangle_\Gamma
	&=
	\Re
	\left[
	\begin{pmatrix}
	\bvar \\	\bpsi
	\end{pmatrix},
	\Cald(s)
	\begin{pmatrix}
	\bvar \\
	\bpsi
	\end{pmatrix}
	\right]_\Gamma+
	\Re \langle \bvar, \Z(s)\bvar\rangle_\Gamma
	\\
	&\ge \frac 1{c_\Gamma^2} \,  \frac{\Re s}{|s|^2+1}  \left(\bigl\|{\bvar}\bigr\|^2_{\mXG}+\bigl\|\bpsi\bigr\|^2_{\mXG}\right) +
	c_\sigma^{(\Z)}  \frac{\Re s}{|s|^2}  \, \bigl| \bvar \bigr| _{\mVG}^2
	\\
	&\ge c_{\sigma} \, \frac{\Re s}{|s|^2}  \left(\bigl\|{\bvar}\bigr\|^2_{\mVG}+\bigl\|\bpsi\bigr\|^2_{\mXG}\right),
	\end{align*}
	which is the stated result.
	\qed
\end{proof}

From the previous two lemmas we obtain the following result.

\begin{theorem} [Well-posedness of the time-harmonic boundary integral equation]
\label{prop:bie}
For $\Re s>\sigma_0\ge 0$, the boundary integral equation~\eqref{bie-weak-A}, with the boundary operator $\A(s):\mVG\times\mXG \rightarrow \mVG'\times \mXG'$ defined by \eqref{Asoperator}, has a unique solution $(\wvarphi,\wpsi)\in\mVG\times \mXG$, and
\begin{equation} \label{phi-psi-bound}
\norm{
	\begin{pmatrix}\wvarphi\\ \wpsi\end{pmatrix}
}_{\mVG\times \mXG}
\le 
 C_\sigma \,\dfrac{\abs{s}^2}{\Re  s}\,\norm{\wginc}_{\mVG'}.
\end{equation}
The constant $C_\sigma$ only depends polynomially on $\sigma^{-1}$ and on the boundary $\Gamma$ via the norm of the tangential trace operator.
\end{theorem}

\begin{proof}
By the Lax--Milgram theorem, Lemmas~\ref{lem:A-boundedness} and~\ref{lem:A-coercivity} yield that $\A(s)$ is invertible and its inverse is bounded,
for $\Re s \ge \sigma>\sigma_0$, by
\begin{equation}
\label{boundAm1}
\norm{\A(s)^{-1}}_{\mVG\times\mXG \leftarrow \mVG'\times \mXG'}\le C_\sigma \,\dfrac{\abs{s}^2}{\Re  s}.
\end{equation}
This gives the result.
\qed
\end{proof}

\begin{remark} \label{rem:norms}
In Lemmas~\ref{lem:EN} and~\ref{lem:absorb} we have $\delta^{1/2} \| \bphi \|_{\bL^2(\Gamma)} \le \| \bphi \|_\mVG$ for all $\bphi\in\mVG$. This implies that for a tangential vector field $\wginc\in \bL^2(\Gamma)$,
\begin{align*}
\norm{\wginc}_{\mVG'} &= \sup_{\| \bphi \|_\mVG = 1} \langle \bphi,\wginc \rangle_\Gamma = \sup_{\| \bphi \|_\mVG = 1} ( \bphi,\wginc )_\Gamma
\\
&\le \sup_{\| \bphi \|_{\bL^2(\Gamma)} \le \delta^{-1/2}} (  \bphi,\wginc )_\Gamma = \delta^{-1/2} \,\norm{\wginc}_{\bL^2(\Gamma)}.
\end{align*}
On the other hand, we have $\| \bphi \|_\mXG \le \| \bphi \|_\mVG$  for all $\bphi\in\mVG$. If  $\wginc$ is in $\mXG$,  we therefore obtain
$$
\norm{\wginc}_{\mVG'} \le \norm{\wginc}_{\mXG}
$$
without any dependence on the small parameter $\delta$. We do have $\wginc\in\mXG$ in the case where
$\wginc = -\widehat{\E}\inc_T -\Z(s)\gaT \widehat{\B}\inc$, cf.~\eqref{eq:def g inc}, for a sufficiently regular boundary $\Gamma$ and
sufficiently regular fields ${\widehat \E}\inc$ and ${\widehat \B}\inc$ for $\Z(s)$ in the situations of Lemmas~\ref{lem:EN} and~\ref{lem:absorb}.
\end{remark}

%
\subsection{Well-posedness of time-harmonic scattering with generalized impedance boundary conditions}

Using the above properties, we prove the following result.

\begin{theorem}[Well-posedness of the time-harmonic scattering problem] \label{th:time-harmonic-well-posedness}
For  $\Re s>\sigma_0 \ge 0$, consider the time-harmonic scattering problem \eqref{TH-MW1}--\eqref{TH-MW2} (with the normalization $\eps\mu=1$)
	under the generalized impedance boundary condition \eqref{gibc-weak-harmonic}, with $\Z(s)$ satisfying conditions \eqref{eq:pol_bound}--\eqref{eq:positive_type} and with $\widehat{\g}^\mathrm{inc} \in \mVG'$.

	(a) This problem has a solution
	$(\widehat{\E},\widehat{\B}) \in \H(\curl,\Omega) \times \H(\curl,\Omega)$ given by the
	representation  formulas \eqref{eq:time-harmonic-kirchhoff-E}--\eqref{eq:time-harmonic-kirchhoff-H}. The tangential traces
 are  uniquely determined by the solution
	$(\wvarphi,\wpsi)=(\gaT\widehat{\B},-\gaT \widehat{\E})\in \mVG \times \mXG$
of the system of boundary integral equations of Theorem~\ref{prop:bie}.

	(b) The electromagnetic fields are bounded by
		\begin{align*}
		\| \widehat{\E} \|_{\H(\curl,\Omega)} + \|  \widehat{\B} \|_{\H(\curl,\Omega)}
	\le 
	C_\sigma \dfrac{\abs{s}^3}{(\Re s)^{3/2}}\norm{\widehat{\g}^\mathrm{inc}}_{\mVG'},
	\end{align*}
	where $C_\sigma$ depends on $\sigma$, on $c_\sigma$ of \eqref{eq:positive_type}, and on $\Gamma$ through norms of tangential trace operators, but is independent of $\eps$ and $\mu$ with $\eps\mu=1$ and, in the case of the impedance operators \eqref{eq:thin_layer_ord1}--\eqref{eq:gibc_absorbing_ord2}, independent of the small parameter $\delta$ (but see Remark~\ref{rem:norms}).
\end{theorem}


\begin{proof}
%
	By Theorem~\ref{prop:bie}, the boundary integral equation \eqref{bie-A} has a unique solution $(\wvarphi,\wpsi)\in \mVG\times\mXG$, which is bounded by  \eqref{phi-psi-bound}.

	We are now in the situation of Lemma~\ref{lem:transmission}: The representation formulas \eqref{eq:time-harmonic-kirchhoff-E3}--\eqref{eq:time-harmonic-kirchhoff-H3} define ${\widehat{\E},\widehat{\B} \in \H(\curl,\mathbb{R}^3\setminus\Gamma)}$, which solve the transmission problem \eqref{eq:transmis-1}--\eqref{eq:transmis-4}.
	Furthermore,  upon expressing $(\wvarphi,\wpsi)$ in terms of $(\widehat{\E},\widehat{\B})$ by means of \eqref{eq:transmis-3}--\eqref{eq:transmis-4}, the fundamental identity \eqref{eq:calderon-jump} of the Calder\'on operator implies the identity
	\begin{align}
	\nonumber
	\Bimp(s)\begin{pmatrix}\wvarphi\\ \wpsi\end{pmatrix}
	&=
	\Cald(s)\begin{pmatrix}\wvarphi\\ \wpsi\end{pmatrix}
	- \dfrac{1}{2}
	\begin{pmatrix}\wpsi\\ \wvarphi\end{pmatrix}
	\\ &
	=
	\begin{pmatrix}
	\avg{\gaT\widehat{\E}} \\ \avg{\gaT\widehat{\B}}
	\end{pmatrix}
	-\dfrac{1}{2}
	\begin{pmatrix}-\jmp{\gaT \widehat \E} \\ \jmp{\gaT \widehat \B}\end{pmatrix}
	=
	\begin{pmatrix}
	\gaT^+\widehat{\E} \\
	\gaT^-\widehat{\B}
	\end{pmatrix}.
	\label{eq:Bimp-reduction}
	\end{align}
	By definition, $(\wvarphi,\wpsi)$ solve the weak formulation \eqref{bie-weak} of the boundary integral equation.
	Using the identity above reduces the weak formulation to
	\begin{alignat}{2}\label{eq:gibc-weak-Eplus}
	\big[\bupsilon,\gaT^+\widehat{\E} \big]_\Gamma +  \big\langle \bupsilon, \Z(s)\wvarphi \big\rangle_\Gamma
	&= \big(\bupsilon,\wginc\big)_\Gamma\qquad &&\text{for all }\ \bupsilon \in \mVG, \\
	\big[\bxi,\gaT^-\widehat{\B}\big]_\Gamma&=0
	\qquad &&\text{for all }\ \bxi \in \mXG.\label{eq:Hmin-zero}
	\end{alignat}
	As $\mXG$ coincides with its own dual, we deduce $\gaT^-\widehat{\B}=0$ and hence $\wvarphi=\gaT^+\widehat{\B}$. Therefore, \eqref{eq:gibc-weak-Eplus} implies that $(\widehat{\E},\widehat{\B})|_{\Omega^+}$ are indeed solutions to the time-harmonic Maxwell's equations which satisfy the generalized impedance boundary condition \eqref{gibc-weak-harmonic}.

	Green's formula \eqref{eq: Green} for the solution $\widehat{\E}|_{\Omega^-}$ in the interior domain $\Omega^-$ of \eqref{TH-MW1}--\eqref{TH-MW2} yields
	\begin{equation*}
	\int_{\Omega^-} \bar s\big|\widehat{\E}\big|^2 +s \big|\widehat{\B}\big|^2 \textrm{d} x = -\int_\Gamma (\gaT^- \widehat{\B} \times \normal) \cdot \gaT^- \widehat{\E} \,\textrm{d}\sigma =0,
	\end{equation*}
	which, after taking the real part, gives $\widehat{\E}|_{\Omega^-}=\widehat{\B}|_{\Omega^-}=0 $ and therefore $\gaT^- \widehat{\E} =0$ and $\wpsi=-\gaT^+\widehat{\E}$. This completes the proof of part (a).

With $\gaT^- \widehat{\E} =0$ and $\gaT^-\widehat{\B}=0$ as shown,
we are now in the situation of Lemma~\ref{lem:transmission-0}, which together with the bound of Theorem~\ref{prop:bie} yields the bound of part (b) of the theorem.
	\qed
\end{proof}

\begin{remark} In view of the $L^2$ bound of Lemma~\ref{lem:transmission-0}, we further have the $L^2$ bound
\begin{equation}\label{EH-L2-bound}
		\| \widehat{\E} \|_{\bL^2(\Omega)} + \|  \widehat{\B} \|_{\bL^2(\Omega)}
	\le
	C_\sigma \dfrac{\abs{s}^2}{(\Re s)^{3/2}}\norm{\widehat{\g}^\mathrm{inc}}_{\mVG'}.
\end{equation}
\end{remark}

%
%
%
%

\subsection{Bounds for the time-harmonic potential operators away from the boundary}
\label{section:coercivity and boundedness of Calderon op}
Point evaluations of the potential operators are bounded by means of the following lemma, which yields a more favourable dependence on $s$ for large $\Re s$ than the $\H(\curl,\R^3\setminus\Gamma)$-norm bound of Lemma~\ref{lem:potential-SD-bound}.
On smooth domains similar pointwise bounds already exist, obtained with more straightforward techniques; see \cite[Theorem~4.4~(c)]{BBSV13} for the single layer operator. The proof of the following lemma generalizes the idea given there to the more technical situation of non-smooth boundaries.
\begin{lemma}
	\label{lemma:bound for evaluations}
	The single and double layer potential operators $\mathcal{S}(s),\mathcal{D}(s)$ evaluated at a point $\x\in\mathbb{R}^3\setminus{\Gamma}$ with $d=\dist(\x,\Gamma)>0 $ satisfy the following bounds:
	\begin{align*}
	\abs{\left(\mathcal S (s)\bvar\right) (\x)} &\le C \abs{s}^2e^{-d \Re s} \norm{\bvar}_{\mXG}, \\
	\abs{\left(\mathcal D (s)\bvar\right) (\x)} &\le C \abs{s}^2 e^{-d \Re s} \norm{\bvar}_{\mXG},
	\end{align*}
	for $\Re s\ge \sigma >0$, and for any $\bvar \in \mXG$. The constant $C$ depends only on $\sigma,\x$ and $\Gamma$.
\end{lemma}
\begin{proof}
	Let $\ej$ denote the $j$-th unit vector in $\mathbb{R}^3$, and let $\x\in\Omega$ with $d=\mathrm{dist}(\x,\Gamma)>0$. We then start by analysing the corresponding component of the integral
	\begin{align*}
	\abs{\ej\cdot \int_{\Gamma}G(s,\x-\y) \bvar(\y)\text{d}\y} &=
	\abs{ \int_{\Gamma}G(s,\x-\y)\ej\cdot\bvar(\y)\text{d}\y }
	\\&=
	\abs{ \int_{\Gamma}G(s,\x-\y) \left(\ej\times \normal\right)\cdot\left(\bvar(\y)\times\normal\right)\text{d}\y }
	\\&\le C
	\norm{\gaT\left(G(s,\x-\cdot) \ej\right) }_{\mXG}\norm{\bvar}_{\mXG}
	\\&\le C
	\norm{G(s,\x-\cdot) \ej }_{\H(\curl,\mathbb R^3 \setminus\Omega)}\norm{\bvar}_{\mXG}
	\\&\le C
	\norm{G(s,\x-\cdot)  }_{H^1(\R^3\setminus \Omega)}\norm{\bvar}_{\mXG},
	\end{align*}
	where the estimate on the trace holds due to \cite[Theorem~4.1]{BCS02}.
	The second summand of the single layer operator is estimated more straightforwardly, as
	\begin{align*}
	\abs{\nabla \int_\Gamma G(s,\x-\y)\divG\bvar(\y)\text{d}\y}
	&\le \norm{\nabla G(s,\x-\cdot)}_{\H^{1/2}(\Gamma)}\norm{\divG \bvar(\y)}_{H^{-1/2}(\Gamma)}\\
	&\le \norm{ G(s,\x-\cdot)}_{H^{2}(\R^3\setminus \Omega)}\norm{\bvar}_{\mXG}.
	\end{align*}

	We estimate the double layer potential similarly to the first summand of the single layer, by taking a partial derivative with respect to a coordinate $x_i$ and obtain
	\begin{align*}
	\abs{\partial_{x_i} \ej\cdot \int_{\Gamma}G(s,\x-\y) \bvar(\y)\text{d}\y} &=\abs{ \ej\cdot \int_{\Gamma}\partial_{x_i}G(s,\x-\y) \bvar(\y)\text{d}\y}
	\\ &\le\norm{\partial_{x_i}G(s,\x-\cdot)  }_{H^1(\R^3\setminus\Omega)}\norm{\bvar}_{\mXG}.
	\end{align*}
	Since the $\curl$ operator is a linear combination of partial derivatives, this estimate implies the stated bound for the double layer potential operator.
	\qed
\end{proof}

The proof of the above result immediately implies the following extension for any spatial differential operator.
In particular it implies that, given traces $\wvarphi,\wpsi\in\mXG$, the corresponding (time-harmonic) solution field $\widehat{\E}=-\mathcal S(s)\wvarphi+\mathcal D(s)\wpsi$ is smooth in every point $\x\in\Omega\setminus \Gamma$.

\begin{lemma}\label{lemma:bound for derivatives}
	For every positive integer $k$ and for $j=1,2,3$, we have the following bounds at $\x\in\mathbb{R}^3\setminus{\Gamma}$ with $d=\dist(\x,\Gamma)>0 $ for $\Re s\ge \sigma >0$:
	\begin{equation*}
	\begin{aligned}
	\abs{\left(\partial_{x_j}^k \mathcal S (s)\bvar\right)\! (\x)} &\le C \abs{s}^{2+k} e^{- d \Re s} \norm{\bvar}_{\mXG}, \\
	\abs{\left(\partial_{x_j}^k  \mathcal D (s)\bvar \right) \!(\x)} &\le C \abs{s}^{2+k} e^{- d\Re s} \norm{\bvar}_{\mXG},
	\end{aligned}
	\qquad \text{for all $\bvar \in \mXG$.}
	\end{equation*}
 The constant $C$ depends only on $k,\sigma,\x$ and $\Gamma$. 
\end{lemma}

The following lemma gives a  bound of the potential operators in the operator norm from $\mXG$ to $\H(\curl,\Omega_d)$, where the boundary of $\Omega_d\subset \Omega$ has distance $d$ to $\Gamma$.
\begin{lemma}
	\label{lemma:bound for norms - away}
	Let $\Omega_d= \{\x\in\Omega \mid  \dist(\x,\Gamma )> d\} $ be the domain away from the boundary by at least some fixed distance $d>0$. Then,
	the single and double layer potential operators $\mathcal{S}(s),\mathcal{D}(s)$  satisfy the following bounds:
	\begin{align*}
	\norm{\mathcal S (s)}_{ \H(\curl,\Omega_d)\leftarrow\mXG} &\le Ce^{-d \Re s} \abs{s}^3 , \\
	\norm{\mathcal D (s) }_{\H(\curl,\Omega_d)\leftarrow\mXG } &\le  C e^{-d \Re s} \abs{s}^3  ,
	\end{align*}
	for $\Re s\ge \sigma >0$.  The constant $C$ depends only on $d,\Gamma$ and $\sigma$. 
\end{lemma}
\begin{proof}
	To show the bound for the  single  layer potential, we start with the square of the $\H(\curl,\Omega)$-norm of an image of the  single  layer potential and employ the bounds from Lemma \ref{lemma:bound for derivatives}:
	\begin{align*}
	\norm{\left(\mathcal S (s)\bvar\right) (\x)}^2_{\H(\curl,\Omega_d)}
	&=
	\int_{\Omega_d} \! \abs{\left(\mathcal S (s)\bvar\right) (\x)}^2
	\! + \abs{\left(\curl \mathcal S (s)\bvar\right) (\x)}^2\text{d} \x
	\\&\le
	C_\sigma\norm{\bvar}^2_{\mXG}\abs{s}^6\int_{\Omega_d} e^{-2\dist(\x,\Gamma)\Re s}\text{d} \x
	.
	\end{align*}
	Estimating the last integral then yields the stated result for $\mathcal S (s)$. The result for $\mathcal D (s)$ is obtained by the same argument.
	\qed
\end{proof}

\section{Time-dependent Maxwell's equations with generalized impedance boundary conditions}
\label{section:time dependent Maxwell with GIBC}

The time-harmonic treatment of the previous section extends to the time domain in a direct way, via the passage from the Laplace domain to the time domain  described in Section~\ref{subsec:Z}  (which follows \cite{L94}). This uses the frequency-explicit estimates of Section~\ref{section:time harmonic Maxwell} in an essential way. 
 We start from the time-dependent version of the boundary integral equation \eqref{eq:Bimp_s}, obtained by formally replacing the Laplace transform variable $s$ by the time differentiation operator~$\pt$.

\bigskip\noindent
{\bf Time-dependent boundary integral equation}: {\it
Find time-dependent boundary densities $(\bvar,\bpsi):[0,T]\to \mVG\times\mXG$ (of temporal regularity to be specified later) such that for almost every $t\in [0,T]$ we have for all $(\bupsilon,\bxi)\in\mVG\times\mXG$,}
\begin{align}\label{Aoperator-t}
\left[
\begin{pmatrix}
\bupsilon \\ \bxi
\end{pmatrix},
\Bimp(\pt)
\begin{pmatrix}
\bvar \\
\bpsi
\end{pmatrix}
\right]_\Gamma+
\langle \bupsilon, \Z(\pt)\gaT \bvar \rangle_\Gamma = \langle\bupsilon,\g\inc\rangle_\Gamma.
\end{align}

\bigskip\noindent
Here, $\ginc:[0,T]\to \mVG'$ is given by \eqref{eq:def g inc}, assuming that $\ginc\in \H^m_0(0,T;\mVG')$ with sufficiently large $m$ (to be specified later). We refer to Section~\ref{subsec:Z} for the definition of this spatio-temporal Hilbert space.

With the operators $\A(s):\mVG\times\mXG\rightarrow \mVG'\times\mXG'$ defined by \eqref{Asoperator}, this boundary integral equation is rewritten more compactly as in \eqref{bie-A},
\begin{equation}\label{bie-A-t}
\A(\pt) \begin{pmatrix}
\bvar \\
\bpsi
\end{pmatrix} = \begin{pmatrix}
\ginc \\
0
\end{pmatrix}.
\end{equation}
In view of the bound \eqref{boundAm1} on the operator family $\A(s)^{-1}$ for $\Re s>\sigma_0$, the temporal convolution operator $\A^{-1}(\pt)$ is well-defined by \eqref{Heaviside}, and by the composition rule we have $\A^{-1}(\pt) \A(\pt) = \Id$ and
$\A(\pt) \A^{-1}(\pt) = \Id$. So we have the temporal convolution
\begin{align}\label{eq:time-Ainv}
\begin{pmatrix}
\bvar \\ \bpsi
\end{pmatrix}
=\A^{-1}(\pt)
\begin{pmatrix}
\g\inc \\
0
\end{pmatrix}
\end{align}
as the unique solution of \eqref{bie-A-t}. More precisely, with the argument given above and the bound of \cite[Lemma 2.1]{L94}, i.e. \eqref{sobolev-bound} used for $\A^{-1}(\pt)$ instead of $\Z(\pt)$ and with the exponent $\kappa=2$ by \eqref{boundAm1}, we obtain the following result.

\begin{theorem} [Well-posedness of the time-dependent boundary integral equation]
\label{prop:bie-t} Let $r\ge 0$. For $\ginc\in \H^{r+3}_0(0,T;\mVG')$,
the boundary integral equation~\eqref{Aoperator-t}
has a unique solution $(\bvar,\bpsi)\in \H^{r+1}_0(0,T;\mVG\times \mXG)$, and
\begin{equation} \label{phi-psi-bound}
\norm{
	\begin{pmatrix}\bvar\\ \bpsi\end{pmatrix}
}_{H^{r+1}_0(0,T;\mVG\times \mXG)}
\le 
 C_T\,\norm{\ginc}_{H^{r+3}_0(0,T;\mVG')}.
\end{equation}
Here, $C_T$ depends on $T$ (polynomially if $\sigma_0=0$ in  \eqref{eq:pol_bound}--\eqref{eq:positive_type}) and on the boundary $\Gamma$ via norms of tangential trace operators.
\end{theorem}

With the time-dependent boundary densities $\bvar,\bpsi$ of Theorem~\ref{prop:bie-t}, the scattered wave is obtained by the time-dependent representation formula (assuming here again $\eps\mu=1$)
\begin{align}
\label{eq:time-dependent-representation-E}
\E & = - \mathcal{S}(\pt)\bvar+\mathcal{D}(\pt)\bpsi, \\
\label{eq:time-dependent-representation-H}
\B & = - \mathcal{D}(\pt) \bvar - \mathcal{S}(\pt) \bpsi .
\end{align}
We now give the well-posedness result for the time-dependent scattering problem under the generalized impedance boundary condition,
which follows from the time-harmonic well-posedness result Theorem~\ref{th:time-harmonic-well-posedness}.

\begin{theorem}[Well-posedness of the time-dependent scattering problem]
\label{th:time-dependent-well-posedness} \ Consider the time-dependent scattering problem \eqref{MW1}--\eqref{MW2} (with the normalization $\eps\mu=1$)
	under the generalized impedance boundary condition \eqref{gibc-weak}, with $\Z(s)$ satisfying conditions \eqref{eq:pol_bound}--\eqref{eq:positive_type} with $\kappa\le 1$ and with $\widehat{\g}^\mathrm{inc} \in \H^{r+3}_0(0,T;\mVG')$ for some arbitrary $r\ge 0$.

	(a) This problem has a unique solution
	$$({\E},{\B}) \in \H^r_0(0,T;\H(\curl,\Omega)^2) \cap \H^{r+1}_0(0,T;(\bL^2(\Omega))^2),
	$$
	which is given by the
	representation  formulas \eqref{eq:time-dependent-representation-E}--\eqref{eq:time-dependent-representation-H}. The tangential traces
 are uniquely determined by the  solution of the system of boundary integral equations of Theorem~\ref{prop:bie-t},
	$$(\bvar,\bpsi)=(\mu\gaT{\H},-\gaT {\E})\in H^{r+1}_0(0,T;\mVG \times \mXG).$$

	(b) The electromagnetic fields are bounded by
		\begin{align*}
		\| {\E} \|_{\H^r_0(0,T;\H(\curl,\Omega))} + \| {\B} \|_{\H^r_0(0,T;\H(\curl,\Omega))}
	\le 
	C_T \|\ginc\|_{\H^{r+3}_0(0,T;\mVG')},
	\end{align*}
	and the same bound is valid for the $\H^{r+1}_0(0,T;(\bL^2(\Omega))^2)$ norms.
 Here, $C_T$ depends on $T$ (polynomially if $\sigma_0=0$ in  \eqref{eq:pol_bound}--\eqref{eq:positive_type}) and on the boundary $\Gamma$ via norms of tangential trace operators, but is independent of $\eps$ and $\mu$ with $\eps\mu=1$ and, in the case of the impedance operators \eqref{eq:thin_layer_ord1}--\eqref{eq:gibc_absorbing_ord2}, independent of the small parameter $\delta$.
\end{theorem}

\begin{proof}
We extend $g^\mathrm{inc} \in \H^{r}_0(0,T; \mVG')$ from the interval $(0,T)$ to a function in $\H^{r}(\R; \mVG')$ on the whole real line, with support in $[0,2T]$. The fields $(\E,\B)$ defined by the time-dependent boundary integral equation \eqref{bie-A-t} and the time-dependent representation formulas \eqref{eq:time-dependent-representation-E}--\eqref{eq:time-dependent-representation-H} have the regularity as stated because of \eqref{sobolev-bound} used for the time-harmonic solution operator with the bounds given in Theorem~\ref{th:time-harmonic-well-posedness}, and they satisfy the stated bounds
on every finite interval $(0, \bar T)$, with at most exponential growth in $\bar T$ of the norm with an arbitrary exponent $\sigma_1>\sigma_0$. The Laplace transform $(\widehat \E(s),\widehat \B(s))$ then exists for $\Re s>\sigma_0$, and it is obtained by the
solution of the time-harmonic boundary integral equation \eqref{bie-A} and the time-harmonic representation formulas
\eqref{eq:time-harmonic-kirchhoff-E}--\eqref{eq:time-harmonic-kirchhoff-H}. By Theorem~\ref{th:time-harmonic-well-posedness},
$(\widehat \E(s),\widehat \B(s))$ is the  solution to the time-harmonic scattering problem with the time-harmonic generalized impedance boundary conditions. Taking the inverse Laplace transform then shows that $(\E,\B)$ solve the time-dependent scattering problem \eqref{MW1}--\eqref{MW2}
under the generalized impedance boundary condition \eqref{gibc-weak}. Finally, the uniqueness of the time-dependent solution $(\E,\B)$ follows from the uniqueness of the tangential traces and the well-posedness of the time-dependent exterior Maxwell problem with a given tangential trace.
\qed
\end{proof}


\section{Semi-discretization in time by Runge--Kutta convolution quadrature}
\label{section:time semi-discrete}

%
%
\subsection{Recap: Runge--Kutta convolution quadrature}

Runge--Kutta convolution quadratures will be used here to approximate temporal convolutions $K(\partial_t)g$; cf.~\eqref{Heaviside}. Let us first recall an $m$-stage implicit Runge--Kutta  discretization of the initial value problem $y' = f(t,y)$, $y(0) = y_0$;  see \cite{HairerWannerII}. For a time step $\tau > 0$, the approximations $y^n$ to $y(t_n)$ at time $t_n = n \tau$, and the internal stages $Y^{ni}$ approximating $y(t_n + c_i \tau)$, are obtained from
\begin{equation*}
\begin{aligned}
Y^{ni} &= y^n + \tau \sum_{j = 1}^m a_{ij} f(t_n+c_jh,Y^{nj}), \qquad i =
1,\dotsc,m,\\
y^{n+1} & = y^n + \tau \sum_{j = 1}^m b_j f(t_n+c_jh,Y^{nj}) .
\end{aligned}
\end{equation*}
The method is given by its coefficients
\begin{equation*}
\mathscr{A} = (a_{ij})_{i,j = 1}^m , \quad \cqb = (b_1,\dotsc,b_m)^T,
\quad \text{and} \quad \cqc = (c_1,\dotsc,c_m)^T .
\end{equation*}
The stability function of the Runge--Kutta method is given by $R(z) = 1 + z b^T (\cqI - z \mathscr{A})^{-1} \bone$, where $\bone = (1,1,\dotsc,1)^T \in \R^m$. We always assume that $\mathscr{A}$ is invertible.

Runge--Kutta methods can be used to construct convolution quadrature methods.
Such methods were first introduced in \cite{LubichOstermann_RKcq} in the context of parabolic problems and were studied for wave propagation problems in \cite{BLM11} and subsequently, e.g., in \cite{BanjaiKachanovska,BanjaiLubich2019,BanjaiMessnerSchanz,BanjaiRieder}. Runge--Kutta convolution quadrature was studied  for the numerical solution of some exterior Maxwell problems in \cite{BBSV13,ChenMonkWangWeile} and of an eddy current problem with an impedance boundary condition in \cite{HiptmairLopezFernandezPaganini}. For wave problems, Runge--Kutta convolution quadrature methods such as those based on the Radau IIA methods (see \cite[Section~IV.5]{HairerWannerII}), often enjoy more favourable properties than their BDF-based counterparts, which are more dissipative and cannot exceed order~2 but are easier to understand and slightly easier to implement.

Let $\cqK(s):\cqX\to \cqY$, $\Re s \ge \sigma_0>0$,  be an analytic family of linear operators between Banach spaces $\cqX$ and $\cqY$, satisfying
the bound, for some exponents $\kappa\in\R$ and $\nu\ge 0$,
\begin{equation}\label{KLM-bound}
\| \cqK(s) \|_{\cqY\leftarrow \cqX} \le M_\sigma \frac{|s|^\kappa}{(\Re s)^\nu}, \qquad \Re s \ge \sigma> \sigma_0.
\end{equation}
As in Section~\ref{subsec:Z}, this yields a convolution operator $\cqK(\pt):H^{r+\kappa}_0(0,T;\cqX) \to \cqH^{r}_0(0,T;\cqY)$  for arbitrary real $r$. For functions $\cqg:[0,T]\to \cqX$ that are sufficiently regular (together with their extension by 0 to the negative real half-axis $t<0$),
we wish to approximate the convolution $(\cqK(\pt)\cqg)(t)$ at discrete times $t_n=n\tau$ with a step size $\tau>0$, using a discrete convolution.

To construct the convolution quadrature weights, we use the \emph{Runge--Kutta differentiation symbol}
\begin{equation}
\label{eq:Delta}
\Delta(\zeta) = \Bigl(\mathscr{A}+\frac\zeta{1-\zeta}\bone \cqb^T\Bigr)^{-1} \in \C^{m \times m}, \qquad
\zeta\in\C \hbox{ with } |\zeta|<1.
\end{equation}
This is well-defined for $|\zeta|<1$ if 
$R(\infty)=1-b^T\mathscr{A}^{-1}\bone$ satisfies $|R(\infty)|\le 1$, as is seen from the Sherman--Woodbury formula. Moreover, for A-stable Runge--Kutta methods (e.g. the Radau IIA methods), the eigenvalues of the matrices $\Delta(\zeta)$ have positive real part for $|\zeta|<1$~\cite[Lemma 3]{BLM11}.

To formulate the Runge--Kutta convolution quadrature for $\cqK(\partial_t )\cqg$, we replace the complex argument $s$ in $\cqK(s)$ by the matrix $\Delta(\zeta)/\tau$ and expand
\begin{equation}\label{rkcq-weights}
\cqK\Bigl(\frac{\Delta(\zeta)}\tau \Bigr) = \sum_{n=0}^\infty {\boldsymb W}_n(\cqK) \zeta^n.
\end{equation}
The operators ${\cqW}_n(\cqK):\cqX^m \to \cqY^m$ are used as the convolution quadrature ``weights''.
For the discrete convolution of these operators with a sequence $\cqg=(\cqg^n)$ with $\cqg^n=(\cqg^n_i)_{i=1}^m\in \cqX^m$
we use the notation
\begin{equation}\label{rkcq}
\bigl(\cqK(\pttau) \boldsymb g \bigr)^n = \sum_{j=0}^n {\boldsymb W}_{n-j}(\cqK) \boldsymb g^j \in \cqY^m.
\end{equation}
Given a function $\boldsymb g:[0,T]\to \cqX$, we use this notation for the vectors $\cqg^n = \bigl(\cqg(t_n+c_i\tau)\bigr)_{i=1}^m$ of values of $\cqg$.
The $i$-th component of the vector $\bigl(\cqK(\pttau) \boldsymb g \bigr)^n$ is then an approximation to $\bigl(\cqK(\partial_t)\cqg\bigr)(t_n+c_i\tau)$; see \cite[Theorem 4.2]{BanjaiLubich2019}.

In particular, if $c_m = 1$, as is the case with Radau IIA methods,
the continuous convolution at $t_{n}$ is approximated by the $m$-th, i.e.~last component of the $m$-vector \eqref{rkcq} for $n-1$:
\begin{equation*}
\bigl(\cqK(\partial_t) \cqg \bigr)(t_{n}) \approx   \Bigl[\bigl(\cqK( \pttau) \boldsymb g \bigr)^{n-1}\Bigr]_m \in \cqY. 
\end{equation*}

An essential property is that the composition rule~\eqref{comp-rule} is preserved under this discretization: for two such operator families $\cqK(s)$ and $\cqL(s)$ that map to compatible spaces, we have
\begin{equation}\label{comp-rule-tau}
\cqK(\pttau)\cqL(\pttau)\cqg = (\cqK\cqL)(\pttau)\cqg.
\end{equation}

The following error bound for Runge--Kutta convolution quadrature from~\cite{BLM11}, here directly stated for the Radau IIA methods \cite[Section~IV.5]{HairerWannerII} and transferred to a Banach space setting, will be the basis for our error bounds of the  time discretization.

\begin{lemma}[{\cite[Theorem~3]{BLM11}}]
	\label{lem:RK-CQ}
	Let $\cqK(s):\cqX\to \cqY$, $\Re s > \sigma_0\ge0$,  be an analytic family of linear operators between Banach spaces $\cqX$ and $\cqY$ satisfying
the bound \eqref{KLM-bound} with exponents $\kappa$ and $\nu$.
	Consider the Runge--Kutta convolution quadrature based on the Radau IIA method with $m$ stages. Let $1\le q\le m$ (the most interesting case is $q=m$) and $r>\max(2q-1+\kappa,2q-1,q+1)$. Let $\cqg \in \cqC^r([0,T],\cqX)$ satisfy $\cqg(0)=\cqg'(0)=...=\cqg^{(r-1)}(0)=0$. Then, the following error bound holds at $t_n=n\tau\in[0,T]$:
	\begin{align*}
	&\norm{{ \Bigl[\bigl(\cqK( \pttau) \boldsymb g \bigr)^{n-1}\Bigr]_m-(\cqK(\pt)\cqg)(t_{n})} }_{\cqY}
	\\
	&\quad\quad\quad \le
	C\, M_{1/T}\,\tau^{\min(2q-1,q+1-\kappa+\nu)}
	\left(\|{\cqg^{(r)}(0)}\|_{\cqX}+\int_0^t\|{\cqg^{(r+1)}(t')}\|_{\cqX} \,\mathrm{d}t'
	\right).
	\end{align*}
	The constant C is independent of $\tau$ and $\cqg$ and $M_\sigma$ of \eqref{KLM-bound}, but depends on the exponents $\kappa$ and $\nu$ in \eqref{KLM-bound} and on the final time $T$.
\end{lemma}

%

\subsection{Convolution quadrature for the scattering problem}

Using a Runge--Kutta based convolution quadrature for the semi-discretization in time of the time-dependent boundary integral equation \eqref{bie-A-t} yields the discrete convolution equation
\begin{align}
\label{bie-A-tau}
\A(\pttau)\begin{pmatrix}
\bvar^\tau\\
\bpsi^\tau
\end{pmatrix}
=
 \begin{pmatrix}
\ginc \\
0
\end{pmatrix}.
\end{align}
By the discrete composition rule \eqref{comp-rule-tau}, the solution to this equation is given by the convolution quadrature semi-discretization of the convolution~\eqref{eq:time-Ainv},
\begin{align*}
\begin{pmatrix}
\bvar^\tau\\
\bpsi^\tau
\end{pmatrix}
=
\A^{-1}(\pttau) \begin{pmatrix}
\ginc \\
0
\end{pmatrix}.
\end{align*}
 This formula is extremely useful for the convergence analysis, since it interprets the solution of the discretized boundary integral equation as a mere convolution quadrature, to which we can apply the error bound of Lemma~\ref{lem:RK-CQ} using the bound \eqref{boundAm1} of $\A(s)^{-1}$. (Such an argument was first used in~\cite{L94} for a time-dependent boundary integral equation in the acoustic case.) In particular, no stability issues arise for this time discretization.

The time discretizations of the electromagnetic fields are then obtained by applying the convolution quadrature to the representation formulas \eqref{eq:time-dependent-representation-E}--\eqref{eq:time-dependent-representation-H}:
\begin{align}
\label{eq:time-dependent-representation-E-tau}
\E^\tau & = - \mathcal{S}(\pttau)\bvar^\tau+\mathcal{D}(\pttau)\bpsi^\tau, \\
\label{eq:time-dependent-representation-H-tau}
 \B^\tau & = - \mathcal{D}(\pttau) \bvar^\tau - \mathcal{S}(\pttau) \bpsi^\tau .
\end{align}
Again by the composition rule, this is the convolution quadrature discretization
\begin{equation}\label{U-t}
\begin{pmatrix} \E^\tau \\ \B^\tau
\end{pmatrix}
= \U(\pttau) \ginc
\qquad \text{of}\qquad
\begin{pmatrix} \E \\  \B
\end{pmatrix}
= \U(\pt) \ginc,
\end{equation}
where we have by Theorem~\ref{th:time-dependent-well-posedness} that
$$
\U(s) = \begin{pmatrix}
-\mathcal{S}(s) & \mathcal{D}(s)
\\
-\mathcal{D}(s) & -\mathcal{S}(s)
\end{pmatrix} \A(s)^{-1} \begin{pmatrix} \Id \\ 0 \end{pmatrix}
:
\mVG' \to \H(\curl,\Omega)^2,
$$
for which the bound
$$
\| \U(s) \|_{\H(\curl,\Omega)^2 \leftarrow \mVG'} \le C_\sigma \dfrac{\abs{s}^3}{(\Re s)^{3/2}}, \quad\text{ for }\ \Re s \ge \sigma> \sigma_0\ge 0,
$$
is given in Theorem~\ref{th:time-harmonic-well-posedness}. Moreover, away  from the boundary we obtain by concatenating Lemmas~\ref{lemma:bound for evaluations}--\ref{lemma:bound for norms - away} and Theorem~\ref{prop:bie} that
on $\Omega_d=\{ \x\in \Omega\,:\,\dist(\x,\Gamma)> d\}$ with $d >0$, we have for $\Re s \ge \sigma> \sigma_0\ge 0$ that
$$
\| \U(s) \|_{(\bC^1(\overline\Omega_d)^3)^2\leftarrow \mVG'} +\| \U(s) \|_{\H(\curl,\Omega_d)^2 \leftarrow \mVG'} \le C_\sigma \abs{s}^5 e^{-d \Re s} ,
$$
where the $\bC^1(\overline\Omega_d)$-norm is the maximum norm on continuously differentiable functions and their derivatives on the closure of $\Omega_d$.
 
\begin{remark}
	Discrete fields $(\E^\tau,\B^\tau)$ generated through the discretized representation formulas \eqref{eq:time-dependent-representation-E-tau}--\eqref{eq:time-dependent-representation-H-tau} are, just as their continuous counterparts, divergence-free. This follows from the observation that the generating functions $\sum_{n\ge 0} \E^n \zeta^n, \sum_{n\ge 0} \B^n \zeta^n$ are divergence-free for $\abs{\zeta}<1$.
\end{remark}

Using $\U(s)$ in the role of $K(s)$  and these bounds as \eqref{KLM-bound} in Lemma~\ref{lem:RK-CQ} then directly yields the following result.

\begin{proposition} [Error bound of the semi-discretization in time]
\label{prop:cq} In the situation of Theorem~\ref{th:time-dependent-well-posedness},
consider the Runge--Kutta convolution quadrature based on the Radau IIA method with $m$ stages used for  the semi-discreti\-zation in time \eqref{bie-A-tau} and \eqref{eq:time-dependent-representation-E-tau}--\eqref{eq:time-dependent-representation-H-tau} of the boundary integral equation \eqref{bie-A-t} and the representation formulas \eqref{eq:time-dependent-representation-E}--\eqref{eq:time-dependent-representation-H}, respectively.
For $r>2m+3$, assume that $\ginc \in \bC^r([0,T],\mVG')$, vanishing at $t=0$ together with its first $r-1$ time derivatives. Then, the approximations to the electromagnetic fields $\E^n=\bigl[(\E^\tau)^{n-1}\bigr]_m$ and $\B^n=\bigl[(\B^\tau)^{n-1}\bigr]_m$ satisfy the following error bound of order $m-1/2$ at $t_n=n\tau\in[0,T]$:
\begin{align*}
\norm{ \begin{pmatrix} \E^n - \E(t_n) \\ \B^n - \B(t_n)
\end{pmatrix} } _{\H(\curl,\Omega)^2}
\le C \,\tau^{m-1/2} \, M(\ginc,t_n).
\end{align*}
On $\Omega_d=\{ x\in \Omega\,:\,\dist(x,\Gamma)> d\}$ with $d >0$, there is the full order $2m-1$:
\begin{align*}
\norm{ \begin{pmatrix} \E^n - \E(t_n) \\ \B^n - \B(t_n)
\end{pmatrix} }_{\bigl(\H(\curl,\Omega_d)\cap \bC^1(\overline\Omega_d)^3\bigr)^2}
\le C_d \,\tau^{2m-1} \,M(\ginc,t_n).
\end{align*}
Here, $M(\g,t)=	\|{\g^{(r)}(0)}\|_{\mVG'}
	+\int_0^t \|{\g^{(r+1)}(t')}\|_{\mVG'} \,\mathrm{d}t'
	$.
	The constants $C$ and $C_d$ are independent of $n$, $\tau$ and $g$, but depend on the final time~$T$. $C_d$ additionally depends on the distance $d$. In the case of the impedance operators \eqref{eq:thin_layer_ord1}--\eqref{eq:gibc_absorbing_ord2}, both $C$ and $C_d$ are independent of the small parameter $\delta$.
\end{proposition}

We remark that for acoustic scattering from a sound-soft obstacle, full-order convergence away from the boundary for the Runge-Kutta convolution quadrature time discretization was previously proved in \cite{BLM11}. Proposition~\ref{prop:cq} shows that this favourable error behaviour extends to the electromagnetic scattering with generalized impedance boundary conditions.

\section{Full discretization}
\label{sec:full}

We use a Galerkin approximation of the boundary integral equation \eqref{bie-A-tau} with boundary element spaces $\V_h\subset \mVG$ and $\X_h \subset \mXG$ corresponding to a family of triangulations with mesh width $h$. We choose both $\V_h$ and $\X_h$ to be the Raviart--Thomas boundary element space of order $k\ge 0$ \cite{RT77}, which is defined on the unit triangle
$\widehat{K}$ as reference element by
\begin{align*}
\text{RT}_k(\widehat{K}) = \left\{\x\mapsto \p_1(\x)+p_2(\x) \x \,:\  \p_1\in P_k(\widehat{K})^2\!,\  p_2\in P_k(\widehat{K})\right\},
\end{align*}
where $P_k(\widehat{K})$ is the polynomial space of degree $k$ on $\widehat{K}$.
Raviart--Thomas elements on an arbitrary grid are then obtained in the standard way by piecewise pull-back to the reference element.

We will use the following approximation results, which are obtained from the results collected in Lemma 15 and Theorem 14 of \cite{BH03}; see also the original references \cite[Section~III.3.3]{BreF91} and \cite{BC03}.
Here we use the same notation $\H^p_\times(\Gamma)=\gaT \H^{p+1/2}(\Omega)$ as in \cite{BH03}.

\begin{lemma} \label{lem:RT}
Let $X_h=V_h$ be the $k$-th order Raviart--Thomas boundary element space on $\Gamma$. For every $\bxi\in\mXG\cap \H^{k+1}_\times(\Gamma)$ and $\bupsilon\in\mVG\cap \H^{k+1}_\times(\Gamma)$, with the space $\mVG$ of Lemma~\ref{lem:EN} or Lemma~\ref{lem:absorb}, the best-approximation error is bounded by
\begin{align*}
&\inf_{\bxi_h\in X_h} \| \bxi_h - \bxi \|_{\mXG} \le C h^{k+3/2} \|\bxi \|_{\H^{k+1}_\times(\Gamma)},
\\
&\inf_{\bupsilon_h \in\V_h} \| \bupsilon_h - \bupsilon\|_{\mVG} \le C h^{k+1} \|\bupsilon \|_{\H^{k+1}_\times(\Gamma)}.
\end{align*}
\end{lemma}

\begin{remark}
\label{remark:bestappprox}
We would have expected that the best-approximation error bound in the $\mVG$-norm is $O(h^{k+3/2}+ \delta^{1/2} h^{k+1})$, in analogy to the situation for acoustic  generalized impedance boundary conditions~\cite{BLN20}.
This would, however, require proving the $\mVG$-norm stability of the projection of \cite{BC03} from $\mXG$ to $\X_h$ that was used to show the best-approximation estimate in $\mXG$. If at all possible, this is in any case beyond the scope of this paper.
\end{remark}

The Galerkin approximation of the time-discretized boundary integral equation \eqref{bie-A-tau} on $\V_h\times \X_h$ then reads
\begin{align}\label{bie-A-tau-h}
\left\langle
\begin{pmatrix}
\bupsilon_h \\ \bxi_h
\end{pmatrix}\!,
\A(\pttau)\begin{pmatrix}\bvar_h^\tau \\ \bpsi_h^\tau\end{pmatrix}\right\rangle_\Gamma=\langle\bupsilon_h,\g\inc\rangle_\Gamma \quad \quad \forall \,(\bupsilon_h,\bxi_h) \in (\V_h\times \X_h)^m.
\end{align}
This determines the approximate boundary densities $\bvar_h^\tau=\bigl((\bvar_h^\tau)^n\bigr)$
with $(\bvar_h^\tau)^n$ $=\bigl( (\bvar_h^\tau)^n_i \bigr)_{i=1}^m\in \V_h^m$ and $\bpsi_h^\tau=\bigl((\bpsi_h^\tau)^n\bigr)$
with $(\bpsi_h^\tau)^n=\bigl( (\bpsi_h^\tau)^n_i \bigr)_{i=1}^m\in X_h^m$, which are used to define the approximations to the electromagnetic fields via the time-discrete representation formulas
\begin{align}
\label{eq:time-dependent-representation-E-tau-h}
\E^\tau_h & = - \mathcal{S}(\pttau)\bvar^\tau_h+\mathcal{D}(\pttau)\bpsi^\tau_h, \\
\label{eq:time-dependent-representation-H-tau-h}
\B^\tau_h & = - \mathcal{D}(\pttau) \bvar^\tau_h - \mathcal{S}(\pttau) \bpsi^\tau_h .
\end{align}
We then have the following error bounds for the full discretization, obtained under regularity assumptions that are presumably stronger than necessary.

\begin{theorem} [Error bound of the  full discretization]
\label{thm:error-full} In the situation of Theorem~\ref{th:time-dependent-well-posedness},
consider \\
---\
Runge--Kutta convolution quadrature based on the Radau IIA method with $m\ge 2$ stages used for  the time discreti\-zation \eqref{bie-A-tau} and \eqref{eq:time-dependent-representation-E-tau}--\eqref{eq:time-dependent-representation-H-tau} of the boundary integral equation \eqref{bie-A-t} and the representation formulas \eqref{eq:time-dependent-representation-E}--\eqref{eq:time-dependent-representation-H}, respectively; and\\
---\  Raviart--Thomas boundary elements of order $k$ for the space discretization of  the boundary integral equation~\eqref{bie-A-t}.
\\
For $r>2m+3$, let $\ginc \in \bC^r([0,T],\mVG')$ vanish at $t=0$ together with its first $r-1$ time derivatives. Furthermore, it is
assumed that the solution $(\bvar,\bpsi)$ of the boundary integral equation~\eqref{bie-A-t} is in
$\bC^{10}([0,T],\H^{k+1}_\times(\Gamma)^2)$, vanishing at $t=0$ together with its time derivatives.
\\
Then, the approximations to the electromagnetic fields $\E^n_h=\bigl[(\E^\tau_h)^{n-1}\bigr]_m$ and $\B^n_h=\bigl[(\B^\tau_h)^{n-1}\bigr]_m$ satisfy the following error bound of order $m-1/2$ in time and order $k+1$ in space at $t_n=n\tau\in[0,T]$:
\begin{align*}
\norm{ \begin{pmatrix} \E^n_h - \E(t_n) \\ \B^n_h - \B(t_n)
\end{pmatrix} } _{\H(\curl,\Omega)^2}
\le C \bigl(\tau^{m-1/2}+h^{k+1}\bigr).
\end{align*}
On $\Omega_d=\{ \x\in \Omega\,:\,\dist(\x,\Gamma)> d\}$ with $d >0$, there is the full order $2m-1$ in time:
\begin{align*}
\norm{ \begin{pmatrix} \E^n_h - \E(t_n) \\ \B^n_h - \B(t_n)
\end{pmatrix} }_{\bigl(\H(\curl,\Omega_d)\cap C^1(\overline\Omega_d)^3\bigr)^2}
\le C_d  \bigl(\tau^{2m-1}+h^{k+1}\bigr).
\end{align*}
	The constants $C$ and $C_d$ are independent of $n$, $\tau$ and $h$, but depend on the final time~$T$ and on the regularity of $\ginc$ and $(\bvar,\bpsi)$ as stated. $C_d$ additionally depends on the distance $d$. In the case of the impedance operators \eqref{eq:thin_layer_ord1}--\eqref{eq:gibc_absorbing_ord2}, both $C$ and $C_d$ are independent of the small parameter $\delta$.
\end{theorem}

\begin{proof} We structure the proof into three parts (a)--(c).

(a) {\it (Discretized time-harmonic boundary integral equation).\/} We first consider the time-harmonic boundary integral equation \eqref{bie-weak-A}, for $\Re s \ge \sigma>\sigma_0 \ge 0$.
We denote by $\bL_h(s):\mVG'\to \V_h\times \X_h$ the solution operator $\widehat \g\mapsto (\wvarphi_h,\wpsi_h)$ of the Galerkin approximation in $\V_h\times \X_h$,
\begin{align}\label{bie-A-tau-h-s-g}
\left\langle
\begin{pmatrix}
\bupsilon_h \\ \bxi_h
\end{pmatrix}\!,
\A(s)\begin{pmatrix}\wvarphi_h \\ \wpsi_h\end{pmatrix}
\right\rangle_\Gamma=
\langle\bupsilon_h,\widehat \g\rangle_\Gamma
\quad \quad \forall \,(\bupsilon_h,\bxi_h) \in \V_h\times \X_h,
\end{align}
which by the bound of $A(s)$ in Lemma~\ref{lem:A-boundedness}, the coercivity estimate of Lemma~\ref{lem:A-coercivity} and the Lax--Milgram lemma is bounded by
\begin{equation}\label{L-h-bound}
\| \bL_h(s) \|_{\V_h\times \X_h \leftarrow \mVG'} \le \frac 1{c_\sigma} \, \frac{|s|^2}{\Re s} .
\end{equation}
Next we consider the associated Ritz projection $\bR_h(s):\mVG\times\mXG \to \V_h \times \X_h$, which maps $(\wvarphi,\wpsi)\in\mVG\times\mXG$ to $(\wvarphi_h,\wpsi_h)\in \V_h\times \X_h$ determined by
\begin{align*}
\left\langle
\begin{pmatrix}
\bupsilon_h \\ \bxi_h
\end{pmatrix}\!,
\A(s)\begin{pmatrix}\wvarphi_h \\ \wpsi_h\end{pmatrix}
\right\rangle_\Gamma=
\left\langle
\begin{pmatrix}
\bupsilon_h \\ \bxi_h
\end{pmatrix},
\A(s)\begin{pmatrix}\wvarphi \\ \wpsi\end{pmatrix}\right\rangle_\Gamma
\quad \quad \forall \,(\bupsilon_h,\bxi_h) \in \V_h\times \X_h.
\end{align*}
Again by Lemmas~\ref{lem:A-boundedness} and~\ref{lem:A-coercivity} and the Lax--Milgram lemma, this problem has a unique solution $(\wvarphi_h,\wpsi_h)\in \V_h\times \X_h$, and by C\'ea's lemma,
$$
\left\| \begin{pmatrix}\wvarphi_h \\ \wpsi_h\end{pmatrix} - \begin{pmatrix}\wvarphi \\ \wpsi\end{pmatrix} \right\|_{ \mVG\times\mXG}
\le \frac{C_\sigma}{c_\sigma}\left(\frac{|s|^2}{\Re s} \right)^2
\inf_{(\bupsilon_h,\bxi_h)\in \V_h \times \X_h}
\left\| \begin{pmatrix}\bupsilon_h \\ \bxi_h\end{pmatrix} - \begin{pmatrix}\wvarphi \\ \wpsi\end{pmatrix} \right\|_{ \mVG\times\mXG},
$$
where the right-hand side is further bounded by Lemma~\ref{lem:RT}.
We can thus view the associated error operator $\mathcal{E}_h(s) = \bR_h(s)-\Id$ as a bounded operator from $\H^{k+1}_\times(\Gamma)^2$ to $\mVG\times\mXG$ with the bound, for $\Re s \ge \sigma >\sigma_0\ge 0$,
\begin{equation}\label{Err-h-bound}
\| \mathcal{E}_h(s) \|_{\mVG\times\mXG \leftarrow \H^{k+1}_\times(\Gamma)^2} \le \widetilde C_\sigma \frac{|s|^4}{(\Re s)^2}\, h^{k+1}.
\end{equation}

(b) {\it (Error of the spatial semi-discretization).}
The spatial semi-discretization of the time-dependent boundary integral equation \eqref{bie-A-t},
\begin{align}\label{bie-A-h}
\left\langle
\begin{pmatrix}
\bupsilon_h \\ \bxi_h
\end{pmatrix}\!,
\A(\pt)\begin{pmatrix}\bvar_h \\ \bpsi_h\end{pmatrix}\right\rangle_\Gamma=\langle\bupsilon_h,\g\inc\rangle_\Gamma \quad \quad \forall \,(\bupsilon_h,\bxi_h) \in (\V_h\times \X_h)^m,
\end{align}
then has the unique solution
$$
\begin{pmatrix}\bvar_h \\ \bpsi_h\end{pmatrix} = \bL_h(\pt)\ginc = \bR_h(\pt) \begin{pmatrix}\bvar \\ \bpsi\end{pmatrix} ,
$$
where $(\bvar,\bpsi)^\top=\A^{-1}(\pt)(\ginc,0)^\top$ is the solution of \eqref{bie-A-t}.
We abbreviate
$$
\W(s) = \begin{pmatrix} -\mathcal{S}(s) & \mathcal{D}(s) \\
-\mathcal{D}(s) & -\mathcal{S}(s) ,
\end{pmatrix}
$$
and set
\begin{equation}\label{U-h}
\U_h(s) = \W(s) \bL_h(s) : \mVG' \to \H(\curl,\Omega)^2.
\end{equation}
By \eqref{L-h-bound} and Lemma~\ref{lem:transmission}, this is bounded by
\begin{equation}\label{U-h-bound}
\| \U_h(s) \|_{\H(\curl,\Omega)^2 \leftarrow \mVG'} \le \bar C_\sigma\,  \frac{|s|^4}{(\Re s)^2}.
\end{equation}
The spatial semi-discretization of the scattering problem is then obtained as
$$
\begin{pmatrix} \E_h \\ \B_h
\end{pmatrix}
= \U_h(\pt) \ginc.
$$
In view of \eqref{U-t}, its error is
\begin{align*}
\begin{pmatrix} \E_h \\ \B_h
\end{pmatrix} -
\begin{pmatrix} \E \\  \B
\end{pmatrix}
&=
\U_h(\pt)\ginc - \U(\pt)\ginc = \W(\pt)\begin{pmatrix}\bvar_h \\ \bpsi_h\end{pmatrix} - \W(\pt)\begin{pmatrix}\bvar \\ \bpsi\end{pmatrix}
\\
&= \W(\pt)(\bR_h - \Id)\begin{pmatrix}\bvar \\ \bpsi\end{pmatrix} = \W(\pt) \,\mathcal{E}_h(\pt) \!\begin{pmatrix}\bvar \\ \bpsi\end{pmatrix}.
\end{align*}
Using the bound of Lemma~\ref{lem:transmission} for the potential operator $\W(s)$, the bound \eqref{Err-h-bound} for the error operator $\mathcal{E}_h(s)$, and the
bound \eqref{sobolev-bound} (with $\kappa=6$) for their composition, and finally the Sobolev embedding
$H^1(0,T;H)\subset C([0,T],H)$ for any Hilbert space $H$, we obtain for the error of the spatial semi-discretization
\begin{align}\label{error-h}
&\max_{0\le t \le T} \left\| \begin{pmatrix} \E_h(t) \\  \B_h(t)
\end{pmatrix} -
\begin{pmatrix} \E(t) \\  \B(t)
\end{pmatrix}
\right\|_{\H(\curl,\Omega)^2}
\\
\nonumber
&\le
C  \left\| \begin{pmatrix} \E_h \\  \B_h
\end{pmatrix} -
\begin{pmatrix} \E \\  \B
\end{pmatrix}
\right\|_{\H^1_0(0,T;\H(\curl,\Omega)^2)}
\le C_T \, h^{k+1} \left\| \begin{pmatrix}\bvar \\ \bpsi\end{pmatrix} \right\|_{\H^7_0(0,T;\H^{k+1}_\times(\Gamma)^2)}.
\end{align}
Using the same argument with the pointwise bounds away from the boundary given by Lemmas~\ref{lemma:bound for evaluations}
and~\ref{lemma:bound for derivatives}, we further obtain
\begin{equation}\label{error-h-d}
\max_{0\le t \le T} \left\| \begin{pmatrix} \E_h(t) \\  \B_h(t)
\end{pmatrix} -
\begin{pmatrix} \E(t) \\  \B(t)
\end{pmatrix}
\right\|_{\bC^1(\overline\Omega_d)^2}
\le C_T \, h^{k+1} \left\| \begin{pmatrix}\bvar \\ \bpsi\end{pmatrix} \right\|_{\H^8_0(0,T;\H^{k+1}_\times(\Gamma)^2)}.
\end{equation}

(c) {\it (Error of the full discretization).\/}  The total error is
$$
\begin{pmatrix} \E^n_h \\  \B^n_h
\end{pmatrix}
-
\begin{pmatrix} \E^n \\  \B^n
\end{pmatrix}
\quad\
+
\quad\
\begin{pmatrix} \E^n \\  \B^n
\end{pmatrix}
-
\begin{pmatrix} \E(t_n) \\  \B(t_n)
\end{pmatrix}
.
$$
The second difference is the error of the temporal semi-discretization, which is bounded by $O(\tau^{m-1/2})$ in the $\H(\curl,\Omega)^2$ norm in Proposition~\ref{prop:cq}. The first difference
is written as (omitting here the superscript $n-1$ and subscript $m$)
\begin{align*}
\W(\pttau)\mathcal{E}_h(\pttau) \begin{pmatrix}\bvar \\ \bpsi\end{pmatrix} &=
\left(\W(\pttau)\mathcal{E}_h(\pttau) \begin{pmatrix}\bvar \\ \bpsi\end{pmatrix} -
\W(\pt)\mathcal{E}_h(\pt)  \begin{pmatrix}\bvar \\ \bpsi\end{pmatrix} \right)
\\
&+ \
\W(\pt)\mathcal{E}_h(\pt) \begin{pmatrix}\bvar \\ \bpsi\end{pmatrix} .
\end{align*}
The last term is the error of the spatial semi-discretization studied in part~(b),  which is bounded by \eqref{error-h}. The difference written in brackets on the right-hand side is a convolution quadrature error, which can be bounded by Lem\-ma~\ref{lem:RK-CQ}. This gives an $O(h^{k+1})$ error in the $\H(\curl,\Omega)^2$ norm, using that by Lemma~\ref{lem:transmission} and \eqref{Err-h-bound} we have here $M_\sigma\le C_\sigma h^{k+1}$, $\kappa=6$, $\nu=3$ in \eqref{KLM-bound} with $\W(s)\mathcal{E}_h(s)$ in the role of $\K(s)$,
 and choosing $q=2$ and $r=10>2q-1+\kappa$. Note that here $\min(2q-1,q+1-\kappa+\nu)=q-2=0$. Altogether, this yields the stated $O(\tau^{m-1/2}+h^{k+1})$ error bound in the $\H(\curl,\Omega)^2$ norm.

To prove the full-order error bound away from the boundary, we rewrite the error as
$$
\begin{pmatrix} \E^n_h \\  \B^n_h
\end{pmatrix}
-
\begin{pmatrix} \E_h(t_n) \\  \B_h(t_n)
\end{pmatrix}
\quad\
+
\quad\
\begin{pmatrix} \E_h(t_n) \\  \B_h(t_n)
\end{pmatrix}
-
\begin{pmatrix} \E(t_n) \\  \B(t_n)
\end{pmatrix}
.
$$
The second difference is the error of the spatial semi-discretization studied in part~(b). The first difference is a convolution quadrature error for the transfer operator $\U_h(s)$ of \eqref{U-h}:
$$
\begin{pmatrix} \E^n_h \\  \B^n_h
\end{pmatrix}
-
\begin{pmatrix} \E_h(t_n) \\  \B_h(t_n)
\end{pmatrix}
=
\Bigl[\bigl(\U_h(\pttau)\ginc \bigr)^{n-1}\Bigr]_m - \U_h(\pt)\ginc(t_n).
$$
(Estimating this error in the $\H(\curl,\Omega)^2$ norm by Lemma~\ref{lem:RK-CQ} would only give an $O(\tau^{m-1})$ bound instead of the stated $O(\tau^{m-1/2})$ bound, which is why we chose a different path before.)

The full-order error bound away from the boundary
in the $\H(\curl,\Omega_d)$ norm and the $\bC^1(\overline\Omega_d)$ norm then follows from Lemma~\ref{lem:RK-CQ}, using the bounds of
Lemmas~\ref{lemma:bound for evaluations}--\ref{lemma:bound for norms - away} that decay exponentially with $d\,\Re s$, concatenated with the bound \eqref{L-h-bound}. This completes the proof of the error bounds.
\qed
\end{proof}

\section{Implementation and numerical experiments}\label{section:numerics}
We start this final section with a few words on the implementation and then present the results of numerical experiments. The codes which were used to generate the figures in this section are distributed via \cite{Codes}. 
\subsection{Implementation}
The convolution quadrature weights are approximated by discretizing their Cauchy-integral representation with the trapezoidal rule, as already described in \cite{L88b}. This gives the approximation to the weights
\begin{align}\label{cqweights - trapezoidal}
\cqW_n(\cqK)&\approx \dfrac{\rho^{-n}}{L}\sum_{l=0}^{L-1}\cqK\left(\dfrac{ \Delta(\rho\, \zeta_L^{-l})}{\tau}\right)\zeta_L^{nl}, \quad \text{for } 0\le n \le N,
\end{align}
where $\zeta_L=  e^{2\pi i  /L}$. The parameters are chosen such that $L=N+1$ and  $\rho^N=\sqrt{\epsilon}$, where $\epsilon$ denotes the machine precision.

To evaluate the analytic operator family $\cqK(\Delta(\zeta)/\tau)$, for the matrix valued characteristic function $\Delta(\zeta)\in \mathbb{C}^{m\times m}$ at a point $\zeta\in\mathbb{C}$ inside of the unit circle, it is convenient to diagonalize the characteristic function by 
\begin{align*}
	\cqT^{-1}\cqK\left(\Delta(\zeta)\right)\cqT=\cqK\left(\cqT^{-1}\Delta(\zeta)\cqT\right), \quad \text{for invertible } \ \ \cqT \in \mathbb{C}^{m\times m},
\end{align*}
which reduces the evaluation $\cqK(\Delta(\zeta)/\tau)$ to evaluating $\cqK(\cdot)$ at the eigenvalues of $\Delta(\zeta)$.
Plugging the approximations to the quadrature weights into \eqref{rkcq} then gives the scheme
\begin{align*}
\left(\cqK(\pt^\tau)\cqg\right)^n \approx
\dfrac{\rho^{-n}}{L}\sum_{l=0}^{L-1}\zeta_L^{ln} \cqK\left(\dfrac{\Delta(\rho \,\zeta^{-l}_{L})}{\tau} \right) \left[ \sum_{j=0}^N \rho^j \cqg^j \zeta^{-jl}_{L} \right].
\end{align*}
The sums above are realized effectively by the application of FFTs, which leaves the main computational obstacle at the evaluations of the Laplace domain operators $\cqK(\cdot)$ at $mL$ scalar frequencies $s_k \in \mathbb{C}$ for $k = 1, \dotsc, mL$ (i.e.~the collection of eigenvalues of $\cqK(\Delta( \rho\, \zeta_L^{-l} )/\tau)$) with positive real part.
Setting either $\cqK(s)=\A_h(s)^{-1}$ or $\cqK(s)=\U_h(s)$ then  gives schemes to approximate the boundary densities $(\bvar,\bpsi)$ or the electromagnetic fields $\E,\H$, respectively. 

We note that due to symmetric properties of the time-harmonic operators, only half of the Laplace domain evaluations have to be computed \cite{BS09}.

Our numerical experiments were conducted in Python, where the appearing potential and boundary operators were discretized with the library Bempp \cite{Bempp}. As space discretization we choose Raviart--Thomas elements of order $0$ and the arising linear systems were iteratively solved with GMRES. 
The anti-symmetric pairing appearing in the weak formulation \eqref{Asoperator} was realized by choosing corresponding N\'ed\'elec boundary elements as the test space. 

\subsection{Numerical Experiments}
We present two types of numerical experiments.
\begin{itemize}
	\item[-] Convergence experiments, where the errors between the numerical solution and a reference solution are presented, for various mesh sizes and time step sizes, and for different values of $\delta$.
	\item[-] We present the computed numerical solution of a three-dimensional scattering problem with a torus as the obstacle.
\end{itemize}

We test the proposed numerical method with an incidental electric planar wave that solves Maxwell's equations on $\mathbb{R}^3 $, which we set to be
\begin{align}
\label{eq:incident-wave}
	\Einc(t,x)=e^{-50(t-x_3-t_0)^2}e_1,
\end{align}
where $\boldsymbol{e_1} =(1,0,0)^T$ and $t_0=-2$. This incidental wave is scattered from a unit sphere centered around the origin, where we applied the generalized impedance boundary condition corresponding to $\Z(\pt)=\delta \pt^{1/2}$, with $\delta=10^{-1},10$. 
The reference solution is computed using a $0$-th order Raviart--Thomas boundary element space discretization with $23871$ degrees of freedoms and the $3$-stage Radau IIA time discretization of order 5 with $N=2^{10}$ time steps.  

In Figure~\ref{fig:time_conv} and \ref{fig:space_conv} we report on a numerical experiment illustrating the error estimate of Theorem~\ref{thm:error-full}. We plot the error of the point evaluation at ${\boldsymb P  = (2,0,0)}$ between numerical approximation $\E_h^\tau(\boldsymb P,t_n)$ and the reference solution $\E_{ref}(\boldsymb P,t_n)$.

The logarithmic plots in Figure~\ref{fig:time_conv} show the errors against the time step size $\tau$, the lines marked with different symbols correspond to different mesh widths $h$ given in the plot.
Figure~\ref{fig:space_conv} contains the same plots for $\delta = 10^{-1}$ (left) and $\delta = 10$ (right), but reversing the roles of $\tau$ and $h$.

In Figure~\ref{fig:time_conv} we can observe a region where the temporal discretization error dominates, and a region where the spatial discretization error dominates (the curves are flattening out). In the region with small spatial error, we can observe that the error curves match the order of convergence of our theoretical results (note the reference lines), of full classical order $O(\tau^{2m-1})$.

Similarly, for Figure~\ref{fig:space_conv} an analogous description applies but with reversed roles.  Although the error estimates of Theorem~\ref{thm:error-full} are $\delta$-independent, in view of Remark~\ref{remark:bestappprox}, we expect a $\delta$-explicit error bound $O(h^{k+3/2}+ \delta^{1/2} h^{k+1})$.
Figure~\ref{Fig2} reports on the spatial convergence rates with $k=0$. On the left-hand side we can observe that since $\delta$ is small enough the first spatial term dominates in the above error estimate, matching the spatial order $O(h^{3/2})$. On the right-hand side, with a large enough $\delta$ the second term is dominating, matching the spatial order $O(h)$.

\begin{figure}[htbp]
	\label{Fig1}
	\centering
	\includegraphics[trim = 0mm 0mm 0mm 0mm, clip,width=0.8\textwidth,height=0.55\textwidth]{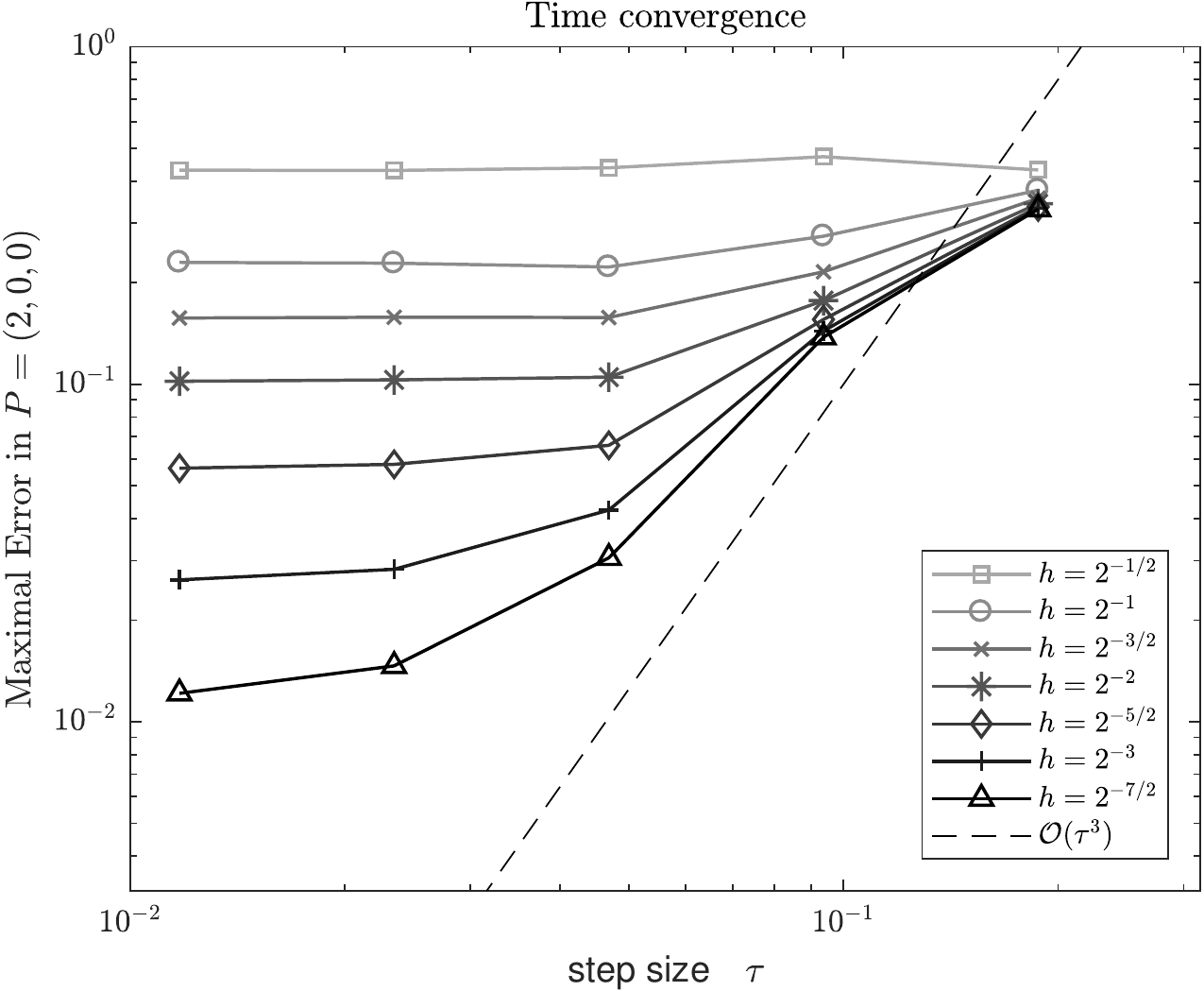}
	\caption{Convergence plot in time for the fully discrete problem, with $\delta = 0.1$}
	\label{fig:time_conv}
\end{figure}

\begin{figure}[htbp]
	\label{Fig2}
	\hspace*{-1.5cm}
	\includegraphics[trim = 0mm 0mm 0mm 0mm, clip,width=1.2\textwidth,height=0.55\textwidth]{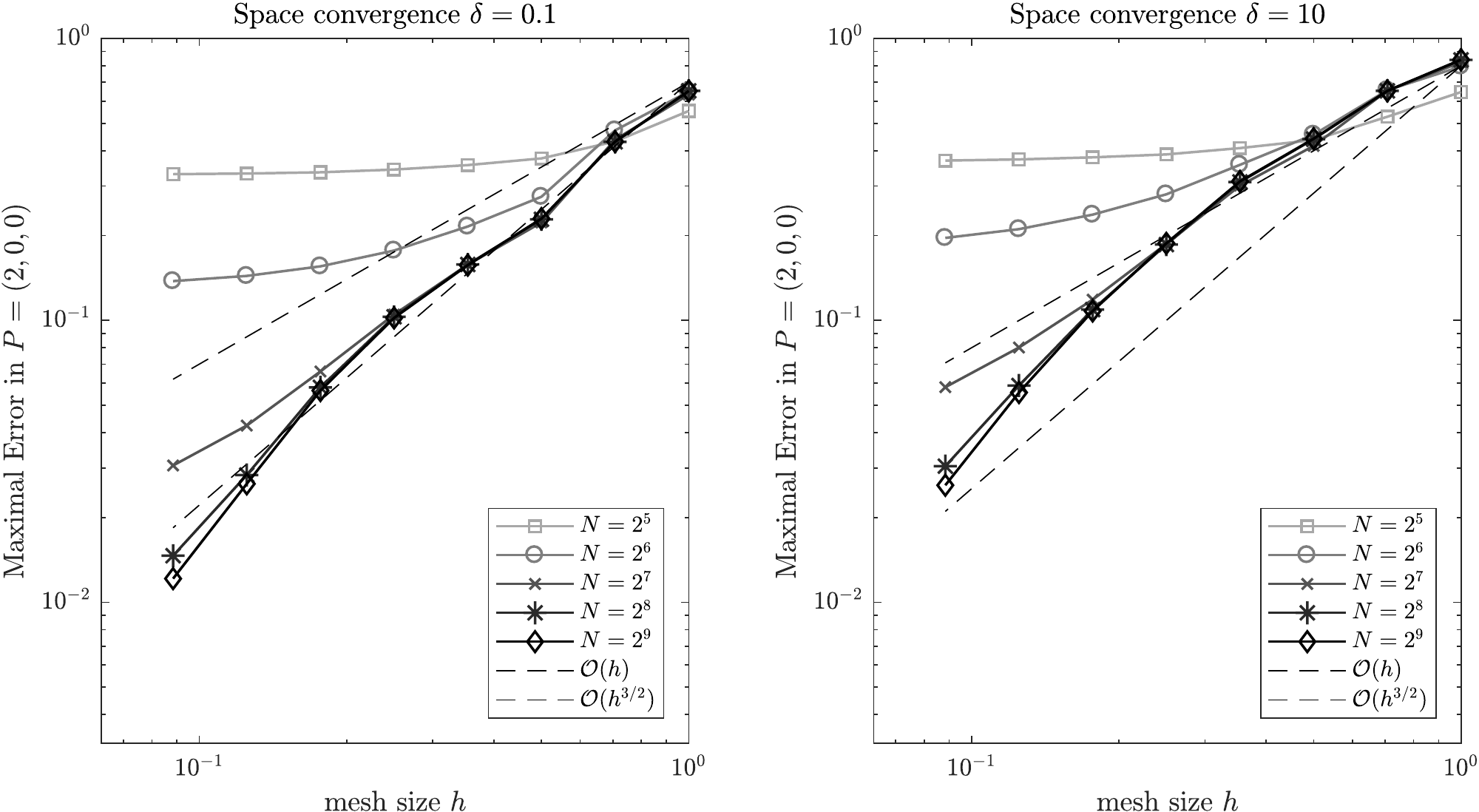}
	\caption{Space convergence plot of the fully discrete system, with varying layer thickness $\delta = 10^{-2}$ (left), and $\delta = 10$ (right).}
	\label{fig:space_conv}
\end{figure}
We conclude our investigations with a visual representation of the scattering arising from a torus with a revolving circle of radius $r=0.2$, where the outer centres lie on a circle of radius $R=0.8$.  The incidental wave \eqref{eq:incident-wave} with $t_0=-1$ is scattered by absorbing boundary conditions corresponding to the impedance operator $\Z(\pt)=\delta\pt^{1/2}$ with $\delta=0.1$ on the torus.

We discretize the described problem in space with $0$-th order Raviart--Thomas boundary elements with $2688$ degrees of freedom and apply convolution quadrature based on the $3$-stage Radau IIA method with $N=100$ time steps.
The left-hand side plot of Figure~\ref{fig:cond} visualizes the frequencies $s_k$ for $k = 1, \dotsc, mL$, at which the Laplace domain operator $\U_h(s_k)$ has to be evaluated. 
The plot on the right-hand side shows condition numbers and norms of the matrix arising from $\A_h(s)$ and its inverse, as one follows the contour depicted before. We observe that the condition number remains relatively mild, which makes iterative solvers accessible to the problem at hand.

 Figure~\ref{fig:frames} then shows the total wave $\E\tot$ on the $x_2=0$ plane at different times.

\begin{figure}[htbp]
	\hspace*{-0cm}
	\includegraphics[trim = 0mm 0mm 0mm 0mm, clip,width=1\textwidth,height=0.5\textwidth]{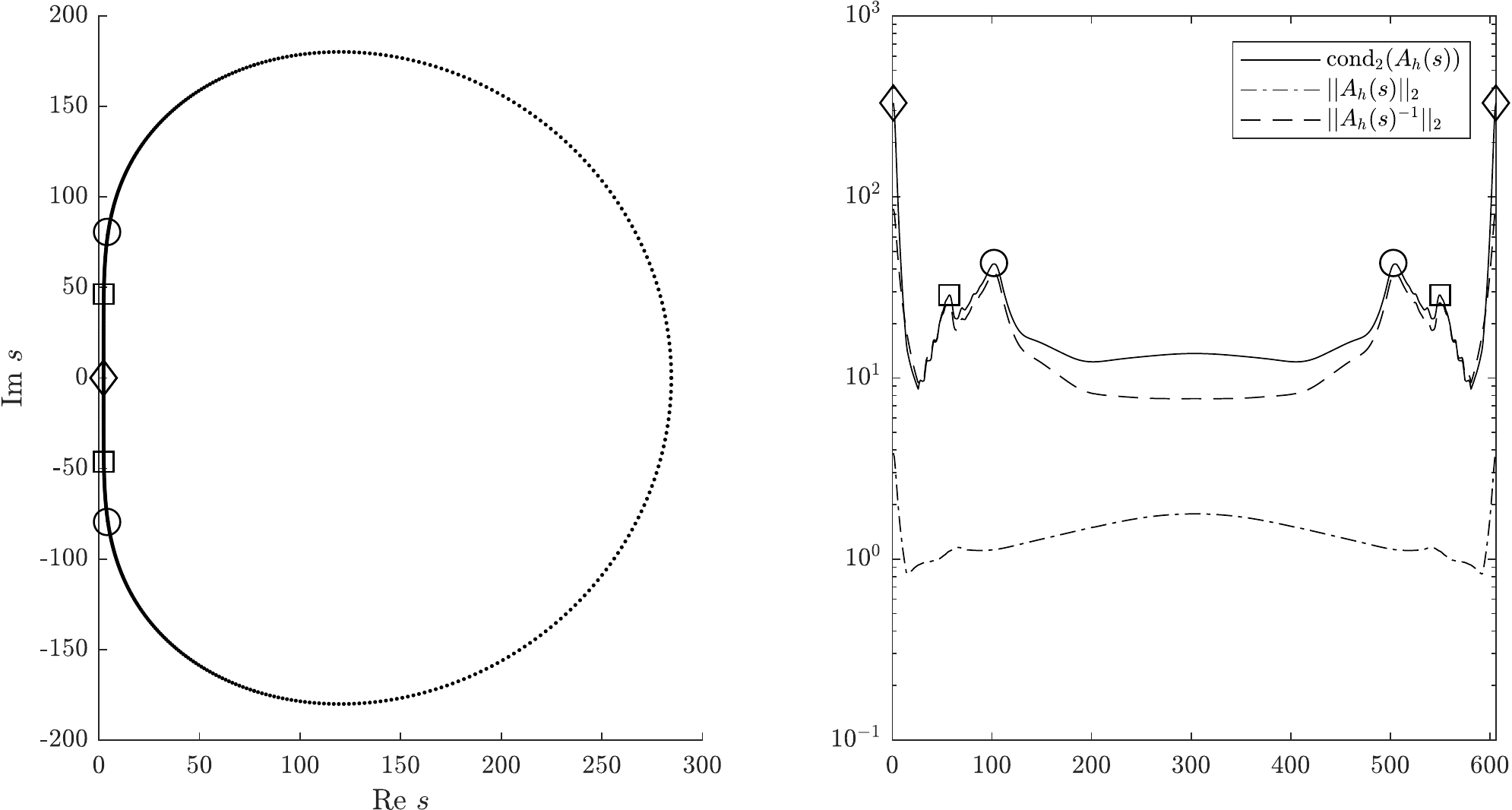}
	\caption{The left-hand side plot shows a plot of the occurring frequencies for the $3$-stage Radau IIA method for $N=100$ and $T=4$. On the right-hand side, the condition numbers and the euclidean norms of the occurring matrices are shown, as they appear when following the integral contour on the left-hand side. The markers on both plots localize the corresponding spikes of the condition numbers on the integral contour.}
	\label{fig:cond}
\end{figure}


\begin{figure}[htp]
	\label{Fig3}
	\centering
	\includegraphics[trim = 1mm 10mm 20mm 11mm, clip,width=0.8\textwidth,height=1.6\textwidth]{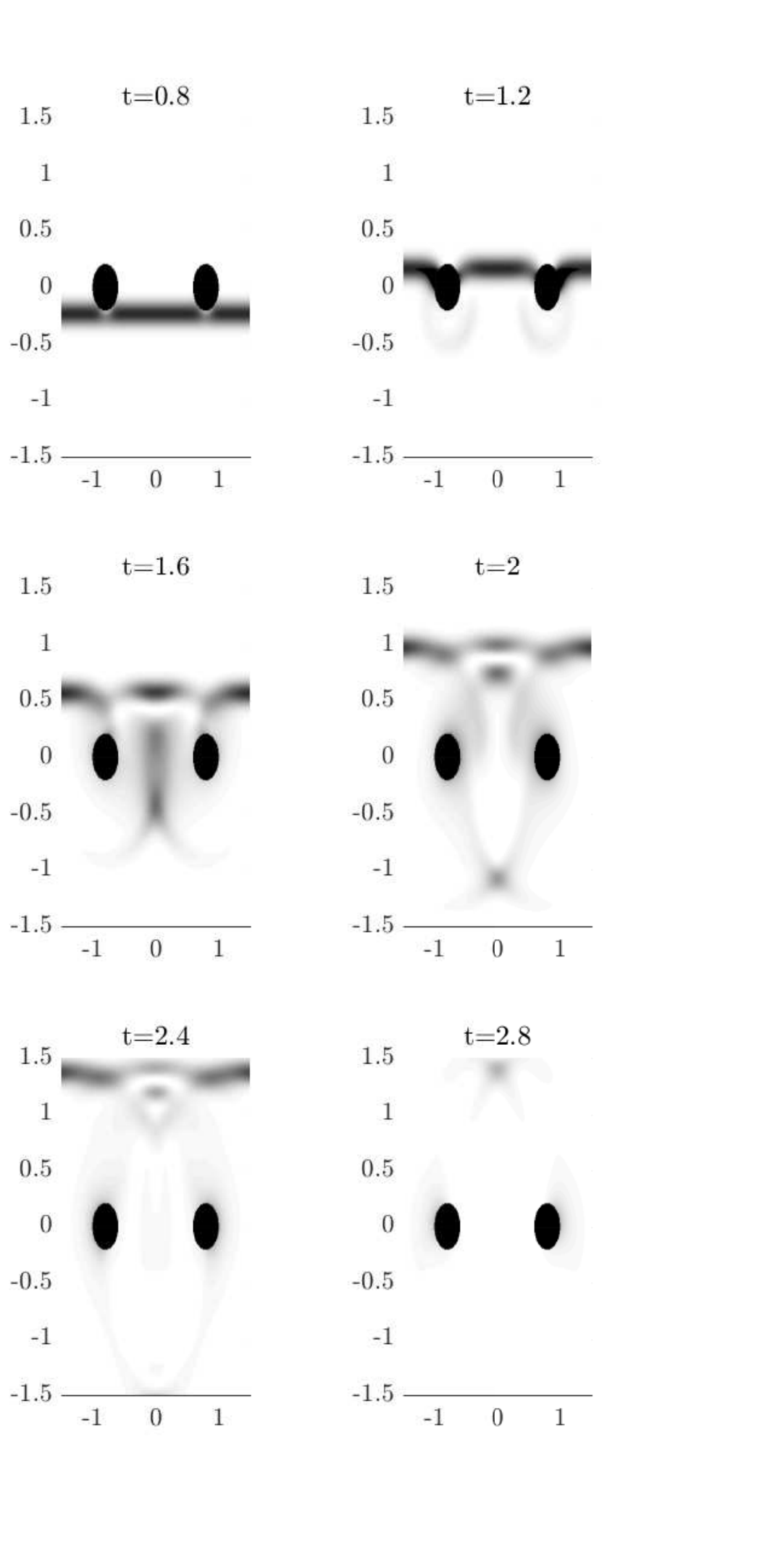}
	\\[-8ex]
	\caption{3D-scattering arising from a torus, visualized at different times. Shown is the $y=0$ plane, through the middle of the scatterer and the boundary condition employed is \eqref{eq:gibc_absorbing_ord1} with $\delta=0.1$}
	\label{fig:frames}
\end{figure}

\section*{Acknowledgement}
We thank Felix Hagemann for his advice regarding the BEM implementation.
The authors are supported by the Deutsche Forschungsgemeinschaft (DFG, German Research Foundation) -- Project-ID 258734477 -- SFB 1173. The work of Bal\'azs Kov\'acs is also supported by the Heisenberg Programme of the Deutsche Forschungsgemeinschaft -- Project-ID 446431602.


\bibliographystyle{abbrv}
\bibliography{Lit}

\begin{thebibliography}{10}

\bibitem{AV96}
A.~Alonso and A.~Valli.
\newblock Some remarks on the characterization of the space of tangential
  traces of {$H (\textrm{rot}; \Omega)$} and the construction of an extension
  operator.
\newblock {\em Manuscripta Math.}, 89(1):159--178, 1996.

\bibitem{AH97}
H.~Ammari and S.~He.
\newblock Generalized effective impedance boundary conditions for an
  inhomogeneous thin layer in electromagnetic scattering.
\newblock {\em J. Electromagn. Waves Appl.}, 11(9):1197--1212, 1997.

\bibitem{AN96}
H.~Ammari and J.-C. N{\'e}d{\'e}lec.
\newblock Sur les conditions d'imp{\'e}dance g{\'e}n{\'e}ralis{\'e}es pour les
  couches minces.
\newblock {\em C. R. Math. Acad. Sci. Paris}, 322(10):995--1000, 1996.

\bibitem{AN99}
H.~Ammari and J.-C. N{\'e}d{\'e}lec.
\newblock Generalized impedance boundary conditions for the {M}axwell equations
  as singular perturbations problems.
\newblock {\em Comm. Partial Differential Equations}, 24(5-6):24--38, 1999.

\bibitem{BBSV13}
J.~Ballani, L.~Banjai, S.~Sauter, and A.~Veit.
\newblock Numerical solution of exterior {M}axwell problems by {G}alerkin {BEM}
  and {R}unge--{K}utta convolution quadrature.
\newblock {\em Numer. Math.}, 123(4):643--670, 2013.

\bibitem{BanjaiKachanovska}
L.~Banjai and M.~Kachanovska.
\newblock Sparsity of {R}unge--{K}utta convolution weights for the
  three-dimensional wave equation.
\newblock {\em BIT}, 54(4):901--936, 2014.

\bibitem{BanjaiLubich2019}
L.~Banjai and C.~Lubich.
\newblock Runge--{K}utta convolution coercivity and its use for time-dependent
  boundary integral equations.
\newblock {\em IMA J. Numer. Anal.}, 39(3):1134--1157, 2019.

\bibitem{BLM11}
L.~Banjai, C.~Lubich, and J.~M. Melenk.
\newblock {R}unge--{K}utta convolution quadrature for operators arising in wave
  propagation.
\newblock {\em Numer. Math.}, 119(1):1--20, 2011.

\bibitem{BLN20}
L.~Banjai, C.~Lubich, and J.~Nick.
\newblock Time-dependent acoustic scattering from generalized impedance
  boundary conditions via boundary elements and convolution quadrature.
\newblock {\em IMA J. Numer. Anal.}, page draa091, 2021.

\bibitem{BLS15}
L.~Banjai, C.~Lubich, and F.-J. Sayas.
\newblock Stable numerical coupling of exterior and interior problems for the
  wave equation.
\newblock {\em Numer. Math.}, 129(4):611--646, 2015.

\bibitem{BanjaiMessnerSchanz}
L.~Banjai, M.~Messner, and M.~Schanz.
\newblock Runge--{K}utta convolution quadrature for the boundary element
  method.
\newblock {\em Comput. Methods Appl. Mech. Engrg.}, 245/246:90--101, 2012.

\bibitem{BanjaiRieder}
L.~Banjai and A.~Rieder.
\newblock Convolution quadrature for the wave equation with a nonlinear
  impedance boundary condition.
\newblock {\em Math. Comp.}, 87(312):1783--1819, 2018.

\bibitem{BS09}
L.~Banjai and S.~Sauter.
\newblock Rapid solution of the wave equation in unbounded domains.
\newblock {\em SIAM J. Numer. Anal.}, 47(1):227--249, 2009.

\bibitem{BreF91}
F.~Brezzi and M.~Fortin.
\newblock {\em Mixed and hybrid finite element methods}, volume~15 of {\em
  Springer Series in Computational Mathematics}.
\newblock Springer-Verlag, New York, 1991.

\bibitem{BC03}
A.~Buffa and S.~H. Christiansen.
\newblock The electric field integral equation on {L}ipschitz screens:
  definitions and numerical approximation.
\newblock {\em Numer. Math.}, 94(2):229--267, 2003.

\bibitem{BCS02}
A.~Buffa, M.~Costabel, and D.~Sheen.
\newblock On traces for {$H(\curl, \Om)$} in {L}ipschitz domains.
\newblock {\em J. Math. Anal. Appl.}, 276(2):845--867, 2002.

\bibitem{BH03}
A.~Buffa and R.~Hiptmair.
\newblock Galerkin boundary element methods for electromagnetic scattering.
\newblock In {\em Topics in computational wave propagation}, volume~31 of {\em
  Lect. Notes Comput. Sci. Eng.}, pages 83--124. Springer, Berlin, 2003.

\bibitem{ChanMonk2015}
J.~F.-C. Chan and P.~Monk.
\newblock Time dependent electromagnetic scattering by a penetrable obstacle.
\newblock {\em BIT}, 55(1):5--31, 2015.

\bibitem{Ch16}
N.~Chaulet.
\newblock The electromagnetic scattering problem with generalized impedance
  boundary conditions.
\newblock {\em ESAIM Math. Model. Numer. Anal.}, 50(3):905--920, 2016.

\bibitem{ChenMonkWangWeile}
Q.~Chen, P.~Monk, X.~Wang, and D.~Weile.
\newblock Analysis of convolution quadrature applied to the time-domain
  electric field integral equation.
\newblock {\em Commun. Comput. Phys.}, 11(2):383--399, 2012.

\bibitem{EN93}
B.~Engquist and J.-C. N{\'e}d{\'e}lec.
\newblock Effective boundary conditions for acoustic and electromagnetic
  scattering in thin layers.
\newblock Technical report, Technical Report of CMAP, 278, 1993.

\bibitem{GLT20}
F.~Z. Goffi, K.~Lemrabet, and T.~Arens.
\newblock Approximate impedance for time-harmonic {M}axwell's equations in a
  non planar domain with contrasted multi-thin layers.
\newblock {\em J. Math. Anal. Appl.}, pages 124--141, 2020.

\bibitem{HJ02}
H.~Haddar and P.~Joly.
\newblock Stability of thin layer approximation of electromagnetic waves
  scattering by linear and nonlinear coatings.
\newblock {\em J. Comput. Appl. Math.}, 143(2):201--236, 2002.

\bibitem{HJN08}
H.~Haddar, P.~Joly, and H.-M. Nguyen.
\newblock Generalized impedance boundary conditions for scattering problems
  from strongly absorbing obstacles: The case of {M}axwell's equations.
\newblock {\em Mathematical Models and Methods in Applied Sciences},
  18(10):1787--1827, 2008.

\bibitem{HairerWannerII}
E.~Hairer and G.~Wanner.
\newblock {\em Solving ordinary differential equations. {II}: {S}tiff and
  differential-algebraic problems}, volume~14 of {\em Springer Series in
  Computational Mathematics}.
\newblock Springer-Verlag, Berlin, 1991.

\bibitem{HiptmairLopezFernandezPaganini}
R.~Hiptmair, M.~L\'{o}pez-Fern\'{a}ndez, and A.~Paganini.
\newblock Fast convolution quadrature based impedance boundary conditions.
\newblock {\em J. Comput. Appl. Math.}, 263:500--517, 2014.

\bibitem{KL17}
B.~Kov\'{a}cs and C.~Lubich.
\newblock Stable and convergent fully discrete interior--exterior coupling of
  {M}axwell’s equations.
\newblock {\em Numer. Math.}, 137(1):91--117, 2017.

\bibitem{LS09}
A.~R. Laliena and F.-J. Sayas.
\newblock Theoretical aspects of the application of convolution quadrature to
  scattering of acoustic waves.
\newblock {\em Numer. Math.}, 112(4):637--678, 2009.

\bibitem{L88b}
C.~Lubich.
\newblock Convolution quadrature and discretized operational calculus. {II}.
\newblock {\em Numer. Math.}, 52(4):413--425, 1988.

\bibitem{L94}
C.~Lubich.
\newblock On the multistep time discretization of linear initial-boundary value
  problems and their boundary integral equations.
\newblock {\em Numer. Math.}, 67(3):365--389, 1994.

\bibitem{LubichOstermann_RKcq}
C.~Lubich and A.~Ostermann.
\newblock Runge--{K}utta methods for parabolic equations and convolution
  quadrature.
\newblock {\em Math. Comp.}, 60(201):105--131, 1993.

\bibitem{Monk_book}
P.~Monk.
\newblock {\em Finite element methods for {M}axwell's equations}.
\newblock Numerical Mathematics and Scientific Computation. Oxford University
  Press, New York, 2003.

\bibitem{Ned01}
J.-C. N{\'e}d{\'e}lec.
\newblock {\em Acoustic and electromagnetic equations: integral representations
  for harmonic problems}.
\newblock Springer, 2001.

\bibitem{Codes}
J.~Nick, B.~Kov\'{a}cs, and C.~Lubich.
\newblock Accompanying codes provided via {G}it{H}ub.
  https://github.com/joerg-nick/{CQM}axwell.

\bibitem{NKL2020}
J.~Nick, B.~Kov\'{a}cs, and C.~Lubich.
\newblock Erratum: Stable and convergent fully discrete interior--exterior
  coupling of {M}axwell’s equations.
\newblock 2020.
\newblock arXiv:1605.04086.

\bibitem{RT77}
P.-A. Raviart and J.~M. Thomas.
\newblock A mixed finite element method for 2nd order elliptic problems.
\newblock In {\em Mathematical aspects of finite element methods ({P}roc.
  {C}onf., {C}onsiglio {N}az. delle {R}icerche ({C}.{N}.{R}.), {R}ome, 1975)},
  pages 292--315. Lecture Notes in Math., Vol. 606, 1977.

\bibitem{S16}
F.-J. Sayas.
\newblock {\em Retarded potentials and time domain boundary integral equations:
  A road map}, volume~50 of {\em Springer Series in Computational Mathematics}.
\newblock Springer, Heidelberg, 2016.

\bibitem{SchH15}
K.~Schmidt and R.~Hiptmair.
\newblock Asymptotic boundary element methods for thin conducting sheets.
\newblock {\em Discrete Contin. Dyn. Syst. Ser. S}, 8(3):619--647, 2015.

\bibitem{Bempp}
W.~{\'S}migaj, T.~Betcke, S.~Arridge, J.~Phillips, and M.~Schweiger.
\newblock Solving boundary integral problems with {BEM++}.
\newblock {\em ACM Trans. Math. Software}, 41(2):1--40, 2015.

\end{thebibliography}

\end{document}